\documentclass{article}
\usepackage{authblk}
\usepackage{lipsum}
\usepackage[latin1]{inputenc}
\usepackage{hyperref}
\usepackage[english]{babel} 
\usepackage{amssymb,amsmath,amsthm,enumerate}
\usepackage{vmargin}
\usepackage{dcolumn}
\usepackage{appendix}
\usepackage{color}
\usepackage{bm}
\usepackage{bbm}
\usepackage{enumitem}
\usepackage{amssymb}
\usepackage{amsthm}
\usepackage{array}
\usepackage[usenames,dvipsnames]{xcolor}
\usepackage{graphicx}
\usepackage{tcolorbox}
\usepackage{amsfonts}
\usepackage{commath}
\usepackage{mathtools}
\usepackage{amssymb}
\usepackage{multirow}
\usepackage{color}
\setmargrb{2.2 cm}{1.7 cm}{2.23 cm}{3.1 cm}
\usepackage{amsmath}
\usepackage[normalem]{ulem}
\usepackage{todonotes}

\newtheorem{theorem}{Theorem}[section]
\newtheorem{definition}[theorem]{Definition}

\newtheorem{proposition}[theorem]{Proposition}
\newtheorem{lemma}[theorem]{Lemma}

\usepackage[useregional]{datetime2}

\newcommand{\bb}{\mathcal{B}}

\usepackage{titlesec}
\titleformat{\subsection}[runin]
  {\normalfont\large\bfseries}{\thesubsection}{1em}{}

   \allowdisplaybreaks
   \begin{document}

 \title{Stationary Boltzmann Equation for Polyatomic Gases in  a slab}
\author[1]{Ki Nam Hong}
\author[2]{Marwa Shahine}
\author[3]{Seok-Bae Yun}
\affil[1,3]{Department of Mathematics, Sungkyunkwan University}
\affil[2]{Department of Mechanical Engineering, Eindhoven University of Technology}



         \pagestyle{myheadings} 

\date{}   

\maketitle

 \begin{abstract}
We consider the existence of steady rarefied flows of polyatomic gas between two parallel condensed phases, where evaporation and condensation processes occur. To this end, we study the existence problem of stationary solutions in a one-dimensional slab for the polyatomic Boltzmann equation, which takes into account the effect of internal energy in the collision process of the gas molecules.
We show that, under suitable norm bound assumptions on the boundary condition functions, there exists a unique mild solution to the stationary polyatomic Boltzmann equation when the slab is sufficiently small.
This is based on various norm estimates - singular estimates, hyperplane estimates - of the collision operator, for which genuinely polyatomic techniques must be employed. For example, in the weighted and singular estimates of the collision operator, we carry out integration with respect to the parameter describing the internal-translational energy distribution, which provides a regularizing effect in the estimate.
 \end{abstract}
\tableofcontents

\section{Introduction}
\subsection{Stationary Boltzmann equation for polyatomic molecules}
The Boltzmann equation is the fundamental model to describe gaseous system at mesoscopic level by
describing the time evolution of the velocity distribution of underlying particles. The original 
Boltzmann equation was derived under the assumption that the molecules are monatomic. This enabled
one to obtain the exact formulation of the pre-collision and post-collision velocities of the molecules
and express the collision operator in an explicit form. However, except for a few exceptions such as helium and
neon, most of the gases in nature are polyatomic. In this regard, Bourgat et al \cite{9}, introduced a polyatomic Boltzmann model by introducing a new variable $I$, which takes into account the intermolecular energy formed due to the rotation and vibration of molecules. In their model, the post-collisional velocities and internal energies of the molecules are expressed in terms of the pre-collisional ones using the parametrization of Borgnakke and Larsen \cite{8}.

In this paper, we study the stationary polyatomic flow arising between two parallel walls on which the evaporation and condensation process occurs, which has been an active topic of research in physics and engineering \cite{aoki2001behavior,Funagane2011,Shigeru,KOSUGE187,a1994,sss,sonedoi,aoki1998,yoshida,kosugex,sone1,sone2,BODIES,boby1996,soneonishi,vladimir,TKT,takata2001,takata1999vapor}.

Mathematically, the problem can be written as the global existence problem for the steady Boltzmann equation describing a single polyatomic gas in an interval of length $\varepsilon$:

\begin{equation}\label{bebc0}
\begin{aligned}
&\hspace{1.8cm}v_{1} \frac{\partial f}{\partial x} = Q(f, f), \qquad x \in[0,\varepsilon], \,\, v\in \mathbb{R}^3 ,\\
&f(0, v,I) =f_L(v,I) , \quad (v_{1}>0), \quad
f(1, v,I) =  f_R( v,I) , \quad (v_{1}<0),
\end{aligned}
\end{equation}
The polyatomic velocity distribution function $f(t,x,v,I)$ is the number density function on position $x\in [0,\varepsilon]$ moving with a translational velocity $v\in\mathbb{R}^3$ with an inter-molecular energy $I \geq 0$ at time $t\geq0$.
$f_L$ and $f_R$ denote the inflow boundary conditions.
The collision operator introduced in \cite{9} takes the following form 
\begin{equation}\label{qff}
	\begin{aligned}
		Q(f, f)(v, I)=\int_{A}\mathcal{B}(v,v_*,I,I_*,r,R,\sigma) \left\{\frac{f(x,v^{\prime},I^{\prime})f(x,v^{\prime}_*,I^{\prime}_*)}{\left(I^{\prime} I_{*}^{\prime}\right)^{\alpha}}-\frac{f(x,v,I) f(x,v_*,I_*)}{\left(I I_{*}\right)^{\alpha}}\right\}\, dB,
	\end{aligned}
\end{equation}
where the domain of integration $A$ is
\[
A= \mathbb{S}^2_{\sigma}\times[0,1]_R\times [0,1]_r \times\{\mathbb{R}\geq0\}_{I_*}\times\mathbb{R}^{3}_{v_*}
\]
where measure $dB$ is defined by
\[
dB=(r(1-r))^{\alpha}(1-R)^{2\alpha+1} R^{\frac{1}{2}} I^{\alpha} I_{*}^{\alpha} \,dR dr d\sigma dI_{*} dv_{*}.
\]
The exponent $\alpha\geq{0}$ is a parameter related to the degrees of freedom of the gas. 
It follows from the consideration of the conservation laws of the momentum and the total energy that  the post-collisional velocity $(v^{\prime},v^{\prime}_*)$ and post-collisional inter-molecular energy $(I^{\prime},I^{\prime}_*)$ are expressed in terms of the pre-collisional velocity $(v,v_*)$ and pre-collisional inter-molecular energy $(I,I_*)$ 
\begin{equation}\label{pre-post velocity}
	\begin{aligned}
		v'=\frac{v+v_*}{2}+\sqrt{{RE}}\,\,\sigma,\qquad v'_*=\frac{v+v_*}{2}-\sqrt{RE}\,\,\sigma,
	\end{aligned}
\end{equation}
and
\begin{equation*}\label{pre-post I}
	\begin{aligned}
        I^{\prime}=r(1-R)E,
		\qquad
        I^{\prime}_*=(1-r)(1-R)E,
	\end{aligned}
\end{equation*}
assuming the molecular mass to be unity. Here $E$ denotes 
total energy of the  molecule in the center of mass reference frame
\begin{equation*}
	E\equiv \frac{1}{4}(v-v_*)^2+I+I_*=\frac{1}{4}(v^{\prime}-v^{\prime}_*)^{2}+I'+I'_*,
\end{equation*}
and the parameter  $R\in[0,1]$ represents the portion allocated to the kinetic energy after collision out of the total energy, and the parameter $r\in[0,1]$ represents the distribution of the post internal energy among the two interacting molecules.
The parameters $r'$ and $R'\in[0,1]$ are defined in the same manner as $r$ and $R$ for the reverse process.

Throughout this paper, we assume that the collision cross-section $\mathcal{B}$ is given by
so-called the ``total energy model" \cite{milana}: 
\begin{equation}\label{eqm}
			\mathcal{B}\left(v, v_{*}, I, I_{*}, r, R,\sigma\right)=E^{\frac{\gamma}{2}},\end{equation}
which is a direct extension of the hard-potential collision kernels of monatomic gases. Note that \eqref{eqm} satisfies the microreversibility conditions
	\begin{equation}\label{reversibility}
		\begin{aligned}
			\bb(v,v_*,I,I_*,r,R,\sigma)&=\bb(v_*,v,I_*,I,1-r,R,-\sigma),\\
			\bb(v,v_*,I,I_*,r,R,\sigma)&=\bb(v',v'_*,I',I'_*,r',R',\sigma),
		\end{aligned}
	\end{equation}
with the following upper and lower bounds introduced in \cite{milana}
	\begin{equation*}
	    \frac{1}{2^{\frac{\gamma}{2}+1}}\;\Big({|v-v_*|}^{{\gamma}}+(I+I_*)^{\frac{\gamma}{2}}\Big)\leq \bb\leq \Big({|v-v_*|}^{{\gamma}}+(I+I_*)^{\frac{\gamma}{2}}\Big)
	\end{equation*}
where $0\leq\gamma\leq1$. 
We mention that several other models for the collision cross section is available, which we listed in the Appendix for reader's convenience.\\

\noindent{\bf $\bullet$ Reformulation of the Problem:}
Upon rescaling, we rewrite the stationary problem posed on an interval of length $\varepsilon$ into the following equivalent problem on the unit interval with the Knudsen number $1/\varepsilon$:
\begin{equation}\label{bebc}
\begin{aligned}
&\hspace{1.8cm}v_{1} \frac{\partial f}{\partial x} =\varepsilon Q(f, f), \qquad x \in[0,1], \,\, v\in \mathbb{R}^3 ,\\
&f(0, v,I) =f_L(v,I) , \quad (v_{1}>0), \quad
f(1, v,I) =  f_R( v,I) , \quad (v_{1}<0),
\end{aligned}
\end{equation}
where the collision operator $Q$ is defined in \eqref{qff}. 

We decompose the collision operator into the gain term $Q^+$ and the loss term $fL(f)$ $$Q(f, f)=Q^{+}(f, f)-f L(f), $$
where 
\begin{equation}\label{q+L}
    Q^{+}(f,f)=\int_{ A }\frac{f^{\prime} f_{*}^{\prime}}{\left(I^{\prime} I_{*}^{\prime}\right)^{\alpha}}\mathcal{B} \, dB,
\quad \text{and \;\;}
L(f)=\int_{ A}{\frac{f_{*}}{(II_*)^{\alpha}}}\mathcal{B} \,dB,\end{equation}
and write the stationary Boltzmann equation \eqref{bebc} as
$$
\frac{\partial f}{\partial x} +\frac{\varepsilon}{v_{1}} f L(f)=\frac{\varepsilon}{v_{1}}Q^{+}(f, f).
$$
Multiplying by the integrating factor

$$
\exp \left(\frac{\varepsilon}{v_1} \int_{y}^{x} L(f)(z, v,I)\, d z\right),
$$
and integrating over $[0,x]$ when $v_1>0$ and over $[x,1]$ if $v_1<0$,  we obtain the following mild form:
 \begin{align}\label{mild solution}
 \begin{split}
     f(x,v,I)&=\exp \left(-\frac{\varepsilon}{ |v_{1}|} \int_{0}^{x} L(f)(y, v,I) \,d y\right)f(0,v,I)\\
  &+\frac{\varepsilon}{ |v_{1}|} \int_{0}^{x}\exp \left(-\frac{\varepsilon}{ |v_{1}|} \int_{y}^{x} L(f)(z, v,I) \,d z\right)Q^+(f,f)\,dy\qquad (v_1>0)\\
    f(x,v,I)&=\exp \left(-\frac{\varepsilon}{ |v_{1}|} \int_{x}^{1} L(f)(y, v,I)\, d y\right)f(1,v,I)\\
  &+\frac{\varepsilon}{ |v_{1}|} \int_{x}^{1}\exp \left(-\frac{\varepsilon}{ |v_{1}|} \int_{x}^{y} L(f)(z, v,I) \,d z\right)Q^+(f,f)\,dy\qquad (v_1<0).
  \end{split}
	\end{align}
We will consider the existence of \eqref{mild solution} from now on.

\subsection{Main Result}\label{mainresult} "We first set up the notation and define our solutions and various norms..\\
\begin{itemize}
\item Throughout the paper, $\varphi$ denote 
 the following weight function:
\begin{equation*}\label{fi}
\varphi(v,I)=\exp\left(a\left(\frac{1}{2}|v|^{2}+I \right)\right) \quad (a>0).
\end{equation*}
\item We define  the weighted singular norm $\|\cdot\|_k$ and the weighted hyperplane norm $\|\cdot\|_P$ as follows:
\begin{align*}
\|f\|_{k} &=\sup _w \int_{\mathbb{R}_I^{+}}\int_{\mathbb{R}_{v}^{3}} \varphi(v,I)\frac{1}{|v-w|^{k}}\|f(\cdot, v,I)\|_{L^{\infty}_x}\, d vdI,\quad k\in \{0,1-\gamma\}, \\
\|f\|_{P} &=\sup _{P} \int_{\mathbb{R}^{+}} \int_{P}\varphi(v,I)\|f(\cdot, v,I)\|_{L^{\infty}_x} \,d \pi_{P,v}dI,
\end{align*}
where $P$ is a hyperplane in $\mathbb{R}_v^3$,
and $d\pi_{P,v}$ is the 2 dimensional measure on $P$.
\item Then, we define
\[
|||f|||=\|f\|_0+\|f\|_{1-\gamma}+\|f\|_P.
\]
\item For simplicity of notation, we use $f_{LR}$ to denote  
$$f_{LR}(v,I) = f_L(v,I)\mathbbm{1}_{v_1>0}+f_R(v,I)\mathbbm{1}_{v_1<0}.$$
\end{itemize}

We define our solutions as follows:
\begin{definition}\label{def1} A non-negative function $f$ 
	is said to be  a mild solution to \eqref{bebc} if $\varphi f\in L^1(\mathbb{R}^3_v\times \mathbb{R}^+_I;L^{\infty}_x(0,1))$
    and satisfies \eqref{mild solution}.
\end{definition}

We are now ready to state our main result. 
\begin{theorem}\label{main theorem}
    Let $1/2\leq\gamma\leq 1$. Assume that the inflow boundary data $f_{LR}\geq0$ satisfies
            \begin{equation}\label{ibc1}
              |||f_{LR}||| < \infty,
              \quad\||v_1|^{-1}(1+(|v|^2+I)^{\gamma/2})f_{LR}\|_0 < \infty.
        \end{equation}

Then, there exists a unique mild solution to \eqref{bebc} for sufficiently small $\varepsilon$.
\end{theorem}

In \cite{maslova2,maslova1}, Maslova proved the existence of a unique mild stationary solution to the Boltzmann equation in a slab for monatomic molecules with the hard sphere collision kernel. Our goal in this paper is to extend this result in two directions: polyatomic molecules and hard potential collision kernels. An interesting observation in the present work is that these two objectives are interconnected.

To apply the Banach fixed point theorem, we need to estimate the collision operator in various norms.
Maslova proved that the weighted $\|\cdot\|_0$ norm estimate of $Q^+$ for the monatomic Boltzmann equation can be controlled by $\|f\|_{-1}\|f\|_0$ in the hard sphere case. Rather unexpectedly, we obtain a stronger estimate of $Q^+$ for the polyatomic Boltzmann equation over the range of $1/2\leq\gamma\leq1$. Namely, the weighted $\|\cdot\|_0$ norm of the polyatomic $Q^+$ can be controlled  using only the $\|\cdot\|_0$ norm of $f$ over the range $1/2\leq\gamma\leq1$.

The key observation is that we can enjoy an interesting polyatomic regularizing effect if we carry out the integration in the parameter $r$ first in the calculation, which describes the distribution of the internal energy in the pre-post collision process. More precisely, in the estimate of $Q^+$ in $\|\cdot\|_0$, we arrive at the following integral (See Lemma \ref{nbb}):
\begin{equation*}
\begin{aligned}
\int_{A\times\mathbb{R}^3\times\mathbb{R}^+} \|{f f_{*}}\|_{L^{\infty}_x}\exp\left(a\left( \frac{|v+v_*|^2}{4}+{RE}+r(1-R)E\right)\!\right)E^{\frac{\gamma}{2}} (1-R)\,  dR dr d\sigma dI_{*} dv_{*}dvdI.
\end{aligned}
\end{equation*}
By carrying out integration in the internal energy distribution parameter $r$ first, we arrive at the integration involving only the pre-collision velocities, which enable one to bound $\|Q^+\|_0$ by $\|f\|_0$ only:
\begin{equation*}
\begin{aligned}
  \|Q^{+}(f,f)\|_0
  &\leq \frac{4\pi}{a}\int \|f f_{*}\|_{L^{\infty}_x}\exp\left(a\left(\frac{|v|^2}{2}+{\frac{|v_*|^2}{2}+I+I_*}\right)\right)\frac{1}{E^{1-\frac{\gamma}{2}}} \,  dI_{*} dv_{*}dvdI.
\end{aligned}
\end{equation*}
Such regularization by the polyatomic parameter, of course, does not arise in the corresponding computation for the monatomic collision operator. This observation is also crucially used in the singular norm estimate of $Q^+$ (see Lemma \ref{singular norm}). Unfortunately, this singular norm estimate is available only for the range $1/2\leq\gamma\leq1$ due to some integrability issue, which is why our results cover only half of the hard potential range. This  will be left for future research.

In \cite{maslova2,maslova1}, the contraction estimate of the head term in the mild form was mistakenly overlooked. In current work, we observe that a velocity growth arise in the missing contraction estimate of the head term, which cannot be controlled by the constant lower bound of the loss term as in \cite{maslova2,maslova1}.
To cancel out this velocity growth and keep the invariance of the solution map, we show that the following stronger lower bound estimate is preserved by the solution map within the solution space (see Lemma \ref{invariance L lower bound}):
$$L(f)\geq a_2\left(1+\left(|v|^2+I\right)^{\frac{\gamma}{2}}\right).$$
We note that this estimate does not hold as an a priori estimate but holds only for elements in the solution space, in contrast to the time-dependent problem for the homogeneous Boltzmann equation or the multiplication operator in the linearized collision operator for the Boltzmann equation near equilibrium.

\subsection{Literature Review: }
The mathematical literature on the stationary Boltzmann equation can be broadly divided into the following three classes, based on the type of solutions considered.

The first class considers the unique mild solutions in a 1-dimensional interval when the interval is sufficiently small, which corresponds to the small-time existence in the time-dependent problem. As in the rescaled problem \eqref{bebc}, such a small interval problem can be rewritten as a stationary flow problem with a large Knudsen number. Research in this direction was initiated by Maslova \cite{maslova2,maslova1} (see also \cite{Gomeshi}), who studied the existence of a unique stationary solution for the monatomic hard sphere Boltzmann equation. In the current work, we extend this result to the polyatomic Boltzmann equation with a hard potential collision kernel.

The second class of stationary results involves the near-equilibrium regime, where the solution is assumed to be a small perturbation of the 
equilibrium. The first result can be traced back to \cite{Asano-Ukai} in which an exterior problem was considered.
 Half-space problems for the linearized Boltzmann equation was studied in  \cite{B-C-N,Chen-Liu-Tong,Deng-Wang-Yu2008,Golse-Perthame-Sulem1988,Guo-Huang-Wang2021,Huang-Jiang-Wang2023,Huang-Wang2022,Liu-Yu,Ukai-Yang-Yu,Wu2021}.
Hydrodynamic limit of stationary Boltzmann equation in the framework of Hilbert expansion was studied in \cite{Esposito-Guo-Kim-Marra2013,E-G-K-M,esposito1994,esposito1995,Guo-Huang-Wang2021,Wu2016}.  
In \cite{E-G-K-M}, the authors studied the incompressible hydrodynamic limit of stationary solutions in a bounded domain.
Regularity of stationary solutions to the Boltzmann equation near equilibrium can be found in \cite{Chen-Kim, Chen-Siam,Guo-Kim-Tonon}.

Studies in the weak solution framework also yielded fruitful results. The existence problem of weak solutions of stationary Boltzmann equation was initiated by a serious of paper by Arkeryd and Nouri \cite{an1995,an20001,an2000,an2002,an2005}. Extensions to gas mixture problem was studied in \cite{b2008,b20081,b20080,b2010}.
For the results on hydrodynamic limit in this direction, we refer to \cite{arkeryd2012ghost,an2006,b2012}.

The mathematical study of the polyatomic Boltzmann equation is still in its early stages. The Fredholm property of the linearized collision operator of the polyatomic Boltzmann equation was established in \cite{Bernhoff2024b,Bernhoff2024a,us,us1}. Cauchy problem for spatially 
homogeneous polyatomic Boltzmann equation was studied in \cite{milana}. Existence and asymptotic stability of the polyatomic Boltzmann equation near equilibrium was obtained in \cite{Duan-Li2023}.
We refer to \cite{Hwang-Yun2022,Park-Yun2016,Park-Yun2019,Son-Yun2023,yun} for mathematical studies of the polyatomic BGK model, which is a relaxational model equation of the polyatomic Boltzmann equation.\\

This paper is organized as follows: In Section 2, various preliminary lemmas are presented, which will be fruitfully used throughout the paper. Section 3 is devoted to the derivation of various necessary norm estimates for the gain term of the collision operator. In Section 4, the solution space is presented, and the solution map is shown to be invariant within this space. The contraction property of the solution map is established in Section 5, which, together with its invariance, leads to the existence of a unique stationary solution by the Banach fixed-point theorem.

\section{Preliminary Lemmas}\label{plemma}
In this Section, we present some preliminary lemmas that will be crucially used throughout the paper. We start with the following estimate of the $\sigma$-integration.

\begin{lemma}\label{lem1}
	For $0\leq\gamma\leq 1$, the following inequality holds
	\begin{equation*}
		\begin{aligned}
			&\int_{\mathbb{S}^2} \frac{1}{|v'+w|^{1-\gamma}} \, d\sigma \leq \frac{8\pi}{(1+\gamma)(RE)^{\frac{1-\gamma}{2}}},
		\end{aligned}
	\end{equation*}
 where $v'$ is the post-collisional velocity defined in \eqref{pre-post velocity}.
\end{lemma}
\begin{proof}
    From \eqref{pre-post velocity}, we see that        \begin{equation*}
		\begin{aligned}
			\int_{\mathbb{S}^2} \frac{1}{|v'+w|^{1-\gamma}} \,d\sigma &= \int_{\mathbb{S}^2} \frac{1}{|\frac{v+v_{*}}{2}+\sqrt{RE}\sigma+w|^{1-\gamma}} \,d\sigma \\
			&= \frac{1}{{(RE)^{\frac{1-\gamma}{2}}}} \int_{\mathbb{S}^2} \frac{1}{|\sigma+ \bar{c}|^{1-\gamma}} \, d\sigma, 
		\end{aligned}
	\end{equation*}
	where  $\bar{c} = \frac{1}{\sqrt{RE}}\left(\frac{v+v_*}{2}+w\right).$ 
    Without loss of generality, we set $\bar{c} = (0,0,|\bar{c}|)$, to find\\ 
	\begin{align*}
		|\sigma + \bar{c}|^{1-\gamma} &= |(\sin\phi\cos\theta, \sin\phi\sin\theta, \cos\phi+|\bar{c}|)|^{1-\gamma} \\
		&= \left( (\sin\phi\cos\theta)^2+ (\sin\phi\sin\theta)^2+ (\cos\phi+|\bar{c}|)^2 \right)^{\frac{1-\gamma}{2}} \\
		&= (1+2|\bar{c}|\cos\phi + |\bar{c}|^2 )^{\frac{1-\gamma}{2}}.
	\end{align*}
	Hence, we have 
	\begin{align*}
		&\frac{1}{(RE)^{\frac{1-\gamma}{2}}} \int_{\mathbb{S}^2} \frac{1}{|\sigma+ \bar{c}|^{1-\gamma}} \, d\sigma \\
		&= \frac{2\pi}{{(RE)^{\frac{1-\gamma}{2}}}} \int_0^{\pi} (1+2|\bar{c}|\cos\phi + |\bar{c}|^2 )^{-\frac{1-\gamma}{2}} \sin\phi \, d\phi \\
		&= \frac{2\pi}{(RE)^{\frac{1-\gamma}{2}}} \int_{-1}^1 (1+ 2|\bar{c}|t+ |\bar{c}|^2)^{-\frac{1-\gamma}{2}} \,dt  \quad (t=\cos\phi) \\
		&= \frac{\pi}{|\bar{c}| (RE)^{\frac{1-\gamma}{2}}} \int_{(|\bar{c}|-1)^2}^{(|\bar{c}|+1)^2} p^{-\frac{1-\gamma}{2}} \, dp\qquad (p=1+ 2|\bar{c}|t+ |\bar{c}|^2) \\
		&= \frac{2\pi}{(1+\gamma)|\bar{c}| (RE)^{\frac{1-\gamma}{2}}} (|\bar{c}|+1)^{\gamma+1}- ||\bar{c}|-1|^{\gamma+1}) \\
		&\leq \frac{8\pi}{(1+\gamma)(RE)^{\frac{1-\gamma}{2}}}.
	\end{align*}
	In the last line, we used  
	$$\frac{(x+1)^{1+\gamma}-|x-1|^{1+\gamma}}{x}\leq 4\qquad (x\geq0)$$
	for any $\gamma$ such that $0\leq 1+\gamma\leq 2$.
\end{proof}

\noindent We now give a symmetry property of the Boltzmann gain operator, which will be used in the norm estimates of the gain operator in the following section.
\begin{lemma}[Symmetry property of the Collision Operator]\label{lem6}
For any nice function $\varphi$ for which the integral
$
 \int_{\mathbb{R}^3\times\mathbb{R}^+} \varphi(v,I)Q^+(f,f)dvdI
$
is well-defined, the following property holds:
\begin{equation*}
\begin{aligned}
  \int_{\mathbb{R}^3\times\mathbb{R}^+} \varphi(v,I)Q^+(f,f)\, dvdI= \int_{A\times \mathbb{R}^3\times\mathbb{R}^+} \varphi\left(v',I'\right){{ff_*} }\mathcal{B}\,  dB dvdI.
\end{aligned}
\end{equation*}
\end{lemma}

\begin{proof}
  We consider the transform:   
\begin{equation*}
    (v,v_*,I,I_*,r,R,\sigma)\mapsto(v',v'_*,I',I'_*,r',R',\sigma'),
\end{equation*}
where the pre-collisional variables are given in terms of the post-collisional variables as
\begin{equation*}
\begin{aligned}
&v=\frac{v'+v'_*}{2}+\sqrt{{R'E}}\sigma',\quad \quad v_*=\frac{v'+v'_*}{2}-\sqrt{{R'E}}\sigma',\\
&\hspace{1cm} I= r'(1-R')E, \hspace{0.75cm}
I_*= (1-r')(1-R')E,\\
&\hspace{1.6cm} r=\frac{I'}{I'+I_*'}, \hspace{0.9cm}
R= \frac{1}{4E}|v'-v'_*|^2.
\end{aligned}
\end{equation*}
Then, an explicit computation shows that the Jacobian of the transform is given by   $J=\frac{R^{\frac{1}{2}}(1-R)}{R'^{\frac{1}{2}}(1-R')}$ as derived in \cite{9,milana,Shahine2023}.
Therefore, we have
\begin{equation*}
\begin{aligned}
  \int \varphi(v,I)&Q^+(f,f)\,dvdI\\
  = &\int \varphi(v,I){\frac{f'f'_{*}}{(I'I'_*)^{\alpha}} }\mathcal{B}  (r(1-r))^{\alpha}(1-R)^{2\alpha+1} R^{\frac{1}{2}} (II_*)^{\alpha}\, dR dr d\sigma dI_{*} dv_{*}dvdI \\
  =&\int \varphi\left(\frac{v'+{v'}_*}{2}+\sqrt{{R'E}}\sigma',r'(1-R')E\right){\frac{f'f'_{*}}{(I'I'_*)^{\alpha}} }\mathcal{B}  (r'(1-r'))^{\alpha} \\
  &\hspace{1cm}(1-R')^{2\alpha+1} {R'}^{\frac{1}{2}}(I'I'_*)^{\alpha}\, dR' dr' d\sigma' dI'_{*} dv'_{*}dv'dI'\\
  =&\int \varphi\left(\frac{v'+v'_*}{2}+\sqrt{{R'E}}\sigma',r'(1-R')E\right){{f(v',I')f(v_*',I'_*)} }\mathcal{B}   \\
  &\hspace{1cm}(r'(1-r'))^{\alpha}(1-R')^{2\alpha+1} {R'}^{\frac{1}{2}}\,dR' dr' d\sigma' dI'_{*} dv'_{*}dv'dI'.
\end{aligned}
\end{equation*}
By dropping the `prime' everywhere we get
\begin{equation*}
\begin{aligned}
  \int \varphi(v,I)Q^+(f,f)\,dvdI
  &= \int \varphi\left(\frac{v+v_*}{2}+\sqrt{{RE}}\sigma,r(1-R)E\right){{f(v,I)f(v_*,I_*)} }\mathcal{B}  \\
  &\hspace{2cm}(r(1-r))^{\alpha}(1-R)^{2\alpha+1} {R}^{\frac{1}{2}} \,dR dr d\sigma dI_{*} dv_{*}dvdI\\
  &= \int \varphi\left(v',I'\right){{ff_*} }\mathcal{B}  (r(1-r))^{\alpha}(1-R)^{2\alpha+1} {R}^{\frac{1}{2}}\, dR dr d\sigma dI_{*} dv_{*}dvdI.
\end{aligned}
\end{equation*} 
\end{proof}

We now provide an upper bound for the collision frequency $L$, which will be used throughout the paper.

\begin{lemma}\label{lemma6rs}
For $0\leq\gamma\leq 1,$ the collision frequency $L$ satisfies
    \begin{align*}
        L(f)\leq 4\pi\frac{e^a}{a}\left(1+\left(|v|^2+I\right)^{\frac{\gamma}{2}}\right)\|f\|_0.
    \end{align*}
\end{lemma}
\begin{proof}
    Using the fact that $\gamma\leq 1$, and using the inequality $x\leq e^x$ for $x\geq 0$, and $1+x+y\leq (1+x)(1+y)$ for $x,y\geq 0$, we have
    \begin{equation*}
        \begin{aligned}
            L(f)&=\int_{ (0,1)^2 \times \mathbb{S}^2\times\mathbb{R}^{3}\times \mathbb{R}^+ }{f_{*}}E^{\frac{\gamma}{2}} (r(1-r))^{\alpha}(1-R)^{2\alpha+1} R^{\frac{1}{2}} \,dR dr d\sigma dv_*dI_*\\
            &\leq 4\pi \int_{ \mathbb{R}^{3}\times \mathbb{R}^+ }{f_{*}}\left(\frac{1}{2}|v|^2+\frac{1}{2}|v_*|^2+I+I_*\right)^{\frac{\gamma}{2}}\, dv_*dI_*.
            \end{aligned}
            \end{equation*}
            Here, we used $\int_{(0,1)^2\times \mathbb{S}^2}(r(1-r))^{\alpha}(1-R)^{2\alpha+1}R^{\frac{1}{2}}\, dRdrd\sigma\leq 4\pi$. Then, we have
            \begin{equation*}
                \begin{aligned}
            L(f)&\leq 4\pi \int_{ \mathbb{R}^{3}\times \mathbb{R}^+ }{f_{*}}\left(\frac{1}{2}|v|^2+\frac{1}{2}|v_*|^2+I+I_*+1\right)^{\frac{\gamma}{2}}\,   dv_*dI_*\\
            &\leq 4\pi \left(1+\frac{1}{2}|v|^2+I\right)^{\frac{\gamma}{2}} \int_{ \mathbb{R}^{3}\times \mathbb{R}^+ }f_* \left(1+\frac{1}{2}|v_*|^2+I_*\right)^{\frac{\gamma}{2}}\, dv_*dI_* \\
            &\leq 4\pi \left(1+\frac{1}{2}|v|^2+I\right)^{\frac{\gamma}{2}} \int_{ \mathbb{R}^{3}\times \mathbb{R}^+ }f_* \left(1+\frac{1}{2}|v_*|^2+I_*\right)\, dv_*dI_* \\
            &\leq 4\pi \left(\frac{e^a}{a}\right) \left(1+\frac{1}{2}|v|^2+I\right)^{\frac{\gamma}{2}} \int_{ \mathbb{R}^{3}\times \mathbb{R}^+ }f_* \exp\left(a\left(\frac{1}{2}|v_*|^2+I_*\right)\right)\, dv_*dI_*.
            \end{aligned}
            \end{equation*}
            By the definition of $\varphi$, we obtain
            \begin{equation*}
                \begin{aligned}
            L(f)&\leq 4\pi \left(\frac{e^a}{a}\right) \left(1+\frac{1}{2}|v|^2+I\right)^{\frac{\gamma}{2}} \int_{ \mathbb{R}^{3}\times \mathbb{R}^+ }\varphi(v_*,I_*) f_* \, dv_*dI_* \\
            &\leq 4\pi \left(\frac{e^a}{a}\right) \left(1+\left(|v|^2+I\right)^{\frac{\gamma}{2}}\right) \|f\|_0. \\
        \end{aligned}
    \end{equation*}
    In the last line, we used $(x+y)^t\leq x^t+y^t$ for $x,y\geq0$ and $0\leq t\leq1$.
\end{proof}

\section{Estimates of $Q^+$ }\label{normest}
In this section, we provide various estimates of the gain part $Q^+$, which will be crucially employed in the following existence proof. 
\subsection{The estimate of $Q^+$ in $\|\cdot\|_0$ norm:}
We start with $\|\cdot\|_0$ estimate of the gain term. We note that, taking integration in $r$ first plays a key role, which is unobserved in the monatomic case. We also remark that our estimate 
improves the corresponding gain term estimate in the monatomic case \cite{maslova1}, in that we only need $\|\cdot\|_0$ norm of $f$ to control $Q^+$ while the monatomic estimate needed the $\|\cdot\|_0$
and $\|\cdot\|_{-1}$ norm of $f$. This improvement came from considering the case $E<1$ and $E\geq1$ separately, and
using the averaging in the internal energy distribution parameter $r$.

\begin{lemma}\label{nbb}
For $0\leq\gamma\leq 1$, the following inequality holds
    \begin{equation*}
        \|Q^{+}(f,f)\|_0 \leq 4\pi \max\left\{\frac{1}{a},1\right\}\|f\|_0^2. 
    \end{equation*} 
\end{lemma}
\begin{proof}
 We consider two cases  $E\leq 1$ and $E>1$ separately.
\begin{enumerate}
    \item \textit{for $E\leq 1$:} We have from $\varphi(v^{\prime},I^{\prime})\leq \varphi(v,I)\varphi(v_*,I_*)$,
 \begin{equation*}
\begin{aligned}
  \|Q^{+}(f,f)\|_0&=\int \varphi(v,I)\|Q^+(f,f)\|_{L^{\infty}_x}\,dvdI\\
  &= \int \|{f f_{*}}\|_{L^{\infty}_x}\varphi(v',I')\mathcal{B}  (r(1-r))^{\alpha}(1-R)^{2\alpha+1} R^{\frac{1}{2}} \, dR dr d\sigma dI_{*} dv_{*}dvdI\\
  &\leq   \int \|{f f_{*}}\|_{L^{\infty}_x}\varphi(v,I)\varphi(v_*,I_*)  \, dR dr d\sigma dI_{*} dv_{*}dvdI\\
    &\leq   \left(\int \|f \|_{L^{\infty}_x}
    \varphi(v,I)dvdI\right)\left(\int
    \|f_{*}\|_{L^{\infty}_x}\varphi(v_*,I_*)  \,  dI_{*} dv_{*}\right)
 \left(   \int dR dr d\sigma\right)\\
  &\leq 4\pi\|f\|_0^2.
\end{aligned}
\end{equation*}
 \item \textit{for $E> 1$:}
 By Lemma \ref{lem6}, we have 
  \begin{equation*}
      \begin{aligned}
      &\|Q^{+}(f,f)\|_0 \\
      &=\int_{\mathbb{R}^3\times\mathbb{R}^+} \varphi(v,I)\|Q^+(f,f)\|_{L^{\infty}_x}\, dvdI\\
  &= \int_{A\times\mathbb{R}^3\times\mathbb{R}^+} \|{f f_{*}}\|_{L^{\infty}_x}\varphi(v',I')\mathcal{B}  (r(1-r))^{\alpha}(1-R)^{2\alpha+1}R^{\frac{1}{2}}\, dR dr d\sigma dI_{*} dv_{*}dvdI \\
    &\leq   \int_{A\times\mathbb{R}^3\times\mathbb{R}^+} \|{f f_{*}}\|_{L^{\infty}_x}\exp\left(a\left( \frac{1}{2}\left(\frac{v+v_*}{2}+\sqrt{RE} \sigma\right)^{2}+r(1-R)E \right)\right)\mathcal{B}   (1-R)\,  dR dr d\sigma dI_{*} dv_{*}dvdI\\
  &\leq \int_{A\times\mathbb{R}^3\times\mathbb{R}^+} \|{f f_{*}}\|_{L^{\infty}_x}\exp\left(a\left( \frac{1}{2}\left(\frac{|v+v_*|}{2}+\sqrt{RE}\right)^{2}+r(1-R)E \right)\!\right)E^{\frac{\gamma}{2}} (1-R) \,dR dr d\sigma dI_{*} dv_{*}dvdI\\
  &\leq \int_{A\times\mathbb{R}^3\times\mathbb{R}^+} \|{f f_{*}}\|_{L^{\infty}_x}\exp\left(a\left( \frac{|v+v_*|^2}{4}+{RE}+r(1-R)E\right)\!\right)E^{\frac{\gamma}{2}} (1-R)\,  dR dr d\sigma dI_{*} dv_{*}dvdI.
\end{aligned}
\end{equation*}
We then integrate in $r$ first, to find
\begin{equation*}
\begin{aligned}
  &\|Q^{+}(f,f)\|_0 \\
  &\leq \frac{1}{a}  \int \|{f f_{*}}\|_{L^{\infty}_x}\exp\left(a\left(\frac{|v+v_*|^2}{4}+{RE}+(1-R)E\right)\right)\frac{1}{E^{1-\frac{\gamma}{2}}} \,  dR  d\sigma dI_{*} dv_{*}dvdI\\
  &= \frac{1}{a}\int \|{f f_{*}}\|_{L^{\infty}_x}\exp\left(a\left(\frac{|v+v_*|^2}{4}+{E}\right)\right)\frac{1}{E^{1-\frac{\gamma}{2}}} \,  dR d\sigma dI_{*} dv_{*}dvdI\\
  &= \frac{4\pi}{a} \int \|{f f_{*}}\|_{L^{\infty}_x}\exp\left(a\left(\frac{|v+v_*|^2}{4}+{\frac{|v-v_*|^2}{4}+I+I_*}\right)\right)\frac{1}{E^{1-\frac{\gamma}{2}}} \,   dI_{*} dv_{*}dvdI\\
  &\leq \frac{4\pi}{a}\int \|f \|_{L^{\infty_x}}\|f_{*}\|_{L^{\infty}_x}\exp\left(a\left(\frac{|v|^2}{2}+{\frac{|v_*|^2}{2}+I+I_*}\right)\right)\frac{1}{E^{1-\frac{\gamma}{2}}} \,  dI_{*} dv_{*}dvdI\\
  & = \frac{4\pi}{a} \int  \|f \|_{L^{\infty}_x}\|f_{*}\|_{L^{\infty}_x}\varphi (v,I)\varphi (v_*,I_*)\frac{1}{E^{1-\frac{\gamma}{2}}} \,   dI_{*} dv_{*}dvdI.
\end{aligned}
\end{equation*}
 Since $E>1$ and $\gamma\leq 1$, we have $\frac{1}{E^{1-\frac{\gamma}{2}}}\leq 1,
$ and therefore we get the desired estimate. 
$$ \|Q^{+}(f,f)\|_0\leq \frac{4\pi}{a}\|f\|_0^2.$$
\end{enumerate}
\end{proof}
\vspace{-\baselineskip}

\subsection{The estimate of $Q^+$ in $\|\cdot\|_P$ norm:} Now we move on to the estimate of $Q$ restricted on a hyperplane. This estimate naturally arise in the invariance estimate later in the existence proof. 
\begin{lemma}\label{planenorm} For $0\leq\gamma\leq 1$, the following inequality holds
    $$\|Q^{+}(f,g)\|_{P} \leq \max\{\pi,2^{2-\gamma}\}(\|f\|_0\cdot \|g\|_0+\|f\|_0\cdot \|g\|_{1-\gamma}).$$
\end{lemma}
\begin{proof}
Let $P$ be given by $av_1+bv_2+cv_3+d=0$. For an arbitrary point $v\in\mathbb{R}^3$, we define $d(v)$ to be the distance from $v$ to the plane $P$:
\[
d(v)=
\frac{av_1+bv_2+cv_3+d}{\sqrt{a^2+b^2+c^2}}
\]
To derive an estimate $Q^+$ restricted on the hyperplane $P$, 
 In order to bound $\|Q^+\|_P$, we introduce a sequence $\left(\frac{\alpha}{\pi}\right)^{\frac{1}{2}}e^{-\alpha d^2 (v)}$ that converges to the Dirac delta function concentrating on $P$ as $\alpha$ goes to infinity.

  \begin{align*}
    &\int \varphi(v,I)\left(\frac{\alpha}{\pi}\right)^{\frac{1}{2}}e^{-\alpha d^2 (v)}Q^{+}(f,g) \, dv dI \\
    &= \frac{1}{2} \int \varphi(v',I') \sqrt{\frac{\alpha}{\pi}} \exp\left(-\alpha\left( \frac{av_1'+bv_2'+cv_3'+d}{\sqrt{a^2+b^2+c^2}}\right)^2\right) {(fg_{*}+f_{*}g)} \\
  &\hspace{3cm}\mathcal{B}\sqrt{R}(r(1-r))^{\alpha}(1-R)^{2\alpha+1}\, dv dv_{*} dI dI_{*} dr dR d\sigma\\
    &\leq \frac{1}{2}\int \varphi(v',I') \sqrt{\frac{\alpha}{\pi}} \exp\left(-\alpha\left( \frac{(a,b,c)\cdot\frac{v+v_*}{2}+\sqrt{RE}(a,b,c)\cdot \sigma+d}{\sqrt{a^2+b^2+c^2}}\right)^2\right) {(fg_{*}+f_{*}g)}  \\
  &\hspace{3cm}\mathcal{B}\sqrt{R}(r(1-r))^{\alpha}(1-R)^{2\alpha+1}\, dv dv_{*} dI dI_{*} dr dR d\sigma
    \\
    &\leq \pi\int \varphi\varphi_*{(fg_{*}+f_{*}g)}\\
    &\;\;\;\sqrt{\frac{\alpha}{\pi}}\int \exp\left( -\alpha\left(\frac{1}{\sqrt{a^2+b^2+c^2}}\left[a\frac{v_1+v_{*}{}_{1}}{2}+b \frac{v_2+v_{*}{}_{2}}{2}+c\frac{v_3+v_{*}{}_{3}}{2}+d\right]+\sqrt{RE}\cos\phi \right)^2 \right)  \\
  &\hspace{3cm}  \sqrt{R}E^{\frac{\gamma}{2}}\sin\phi \, d\phi dr dRdvdv_*dIdI_*
    \end{align*}
Here, $\phi$ denotes the angle between the fixed vector $(a,b,c)$ and $\sigma$. For simplicity, we denote
\[
V=\frac{1}{\sqrt{a^2+b^2+c^2}}\left[a\frac{v_1+v_{*}{}_{1}}{2}+b \frac{v_2+v_{*}{}_{2}}{2}+c\frac{v_3+v_{*}{}_{3}}{2}+d\right]
\]
and carry out integration over $\phi$ first to get
\begin{align*}
&\int \varphi(v,I)\left(\frac{\alpha}{\pi}\right)^{\frac{1}{2}}e^{-\alpha d^2 (v)}Q^{+}(f,g) \, dv dI \\
    &=\sqrt{\pi\alpha}\int \varphi\varphi_*{(fg_{*}+f_{*}g)}\left\{\int_0^{\pi} \exp\left(-\alpha\left(V+\sqrt{RE}\cos \phi\right)^2 \right) \sin \phi \,d\phi\right\}\sqrt{R}E^{\frac{\gamma}{2}} dr dRdvdv_*dIdI_*  \\
    &= \sqrt{\pi\alpha}\int \varphi\varphi_*{(fg_{*}+f_{*}g)}\left\{\frac{1}{\sqrt{RE}}\int_{V-\sqrt{RE}}^{V+\sqrt{RE}} e^{-\alpha t^2} \,dt\right\}\sqrt{R}E^{\frac{\gamma}{2}} dr dRdvdv_*dIdI_*  \\
    &\leq \sqrt{\pi\alpha}\int \varphi\varphi_*{(fg_{*}+f_{*}g)}\left\{\frac{1}{\sqrt{RE}} \int_{\mathbb{R}} e^{-\alpha t^2} \,dt\right\}\sqrt{R}E^{\frac{\gamma}{2}} dr dRdvdv_*dIdI_*  \\
    &= \sqrt{\pi\alpha}\int \varphi\varphi_*{(fg_{*}+f_{*}g)}\left\{\frac{1}{\sqrt{RE}}\sqrt{\frac{\pi}{\alpha}}\right\} \sqrt{R}E^{\frac{\gamma}{2}} dr dRdvdv_*dIdI_*\\
    &= \pi\int \varphi\varphi_*{(fg_{*}+f_{*}g)} \frac{1}{E^{\frac{1-\gamma}{2}}} \,dvdv_*dIdI_*.
\end{align*}
If {$E\geq 1$}, this yields
   \begin{equation}\label{ul1}
\begin{aligned}
    \int \varphi(v,I)\left(\frac{\alpha}{\pi}\right)^{\frac{1}{2}}e^{-\alpha d^2 (v)}Q^{+}(f,g) \,dv dI &\leq \pi \int \varphi \varphi_*(fg_{*}+f_{*}g)\, dv dv_{*} dI dI_{*} \\
    &\leq \pi \|f\|_0\cdot\|g\|_0.
\end{aligned}
\end{equation}  
On the other hand, in the case $E\leq 1$,  we have
\begin{equation}\label{dfd}
    \begin{aligned}
    \int \varphi_{\alpha}(v,I)Q^{+}(f,g) \,dv dI 
    &\leq 2^{1-\gamma}\left(\sup_{\omega}\int \varphi_* g_{*}\frac{1}{|v_*+\omega|^{_{1-\gamma}}}\,dv_*dI_*\right)\int \varphi f \,dvdI\\
  &+2^{1-\gamma}\left(\sup_{\omega}\int \varphi g \frac{1}{|v+\omega|^{1-\gamma}}\,dvdI\right)\int \varphi_* f_*\, dv_*dI_*\\
    &\leq 2^{2-\gamma}\|f\|_0\cdot\|g\|_{1-\gamma} 
\end{aligned}
\end{equation}
Finally, we combine \eqref{ul1} and \eqref{dfd}, 
take the limit $\alpha \rightarrow \infty$, and take the supremum  over $P$ we get the desired estimate:  
\begin{equation*}
    \begin{aligned}
       \|Q^+(f,f)\|_P&=\sup_P\lim_{\alpha\rightarrow \infty} \int \varphi(v,I)\left(\frac{\alpha}{\pi}\right)^{\frac{1}{2}}e^{-\alpha d^2 (v)}Q^{+}(f,f) \,dv dI\\ 
       &\leq {\max\{\pi,2^{2-\gamma}\}(\|f\|_0\cdot\|g\|_0+ \|f\|_0\cdot\|g\|_{1-\gamma}).}
    \end{aligned}
\end{equation*}
\end{proof}
\subsection{The estimate of $Q^+$ in $\|\cdot\|_{1-\gamma}$ norm:}
We note in the previous lemma that $\|\cdot\|_{1-\gamma}$ estimate shows up in the estimate of $\|Q^+(f,f)\|_{P}$. Therefore, we need to estimate the collision operator in $\|\cdot\|_{1-\gamma}$ norm for the sake of invariance properties in the existence proof.
We also remark that, unlike the monatomic Boltzmann case, using the averaging over the internal energy distribution parameter $r$ plays a crucial role in the derivation of the estimate.
\begin{lemma}\label{singular norm}
For $1/2\leq \gamma\leq 1$, $Q^+$ satisfies
\begin{align*}
\|Q^+(f,f)\|_{1-\gamma}\leq \frac{16\pi}{(1+\gamma)^2}\max\left\{1,\frac{1}{a}\right\}\|f\|_0^2
\end{align*}
\end{lemma}
\begin{proof}
When $E\leq 1$, we have
   \begin{equation*}
    \begin{aligned}
     \|Q^{+}(f,f)\|_{1-\gamma}&\leq \int \frac{1}{|v'+w|^{1-\gamma}} \varphi(v',I') ff_{*} \mathcal{B}\, dv dv_{*} dI dI_{*} dr dR d\sigma \\
    &\leq\int \frac{1}{|v'+w|^{1-\gamma}} \varphi \varphi_* ff_{*} E^{\frac{\gamma}{2}}\, dv dv_{*} dI dI_{*} dr dR d\sigma .
   \end{aligned}
   \end{equation*}
Taking integration in $\sigma$ first and applying Lemma \ref{lem1}, we get
 \begin{equation*}
    \begin{aligned}
     \|Q^{+}(f,f)\|_{1-\gamma}
    &\leq \frac{8\pi}{\gamma+1}\int  \varphi \varphi_* ff_{*} E^{\frac{2\gamma-1}{2}}\frac{1}{{R^{\frac{1-\gamma}{2}}}} \,dv dv_{*} dI dI_{*}  dR .
   \end{aligned}
   \end{equation*}
Since $2\gamma-1\geq 0$, we have $E^{\frac{2\gamma-1}{2}}\leq 1$, therefore
 \begin{equation*}
    \begin{aligned}
     \|Q^{+}(f,f)\|_{1-\gamma}
    &\leq \frac{8\pi}{\gamma+1}\int  \varphi \varphi_* ff_{*} \frac{1}{{R^{\frac{1-\gamma}{2}}}} \,dv dv_{*} dI dI_{*}  dR \\
     &\leq \frac{8\pi}{\gamma+1}\int  \varphi \varphi_* ff_{*} \left\{\int\frac{1}{{R^{\frac{1-\gamma}{2}}}}dR\right\} \,dv dv_{*} dI dI_{*}\\
    &\leq \frac{16\pi}{(1+\gamma)^2}\|f\|_0^2.
   \end{aligned}
   \end{equation*}
\noindent In the case $E\geq 1$, we recall \eqref{pre-post velocity} and employ Lemma     \ref{lem1} to compute
     \begin{equation*}
    \begin{aligned}
     &\hspace{-4cm}\|Q^{+}(f,f)\|_{1-\gamma} \\
     &\hspace{-4cm}\leq\int \frac{1}{|v'+w|^{1-\gamma}} \varphi(v',I') ff_{*} \mathcal{B}\sqrt{R}(1-R)^{2\alpha+1} \, dv dv_{*} dI dI_{*} dr dR d\sigma \\
     \end{aligned}
     \end{equation*}
     \begin{equation*}
         \begin{aligned}
    &\leq\int \left\{\int_{\mathbb{S}^2} \frac{1}{|v'+w|^{1-\gamma}} \,d\sigma\right\} \exp\left(a\left(\frac{1}{2}\left(\frac{|v+v_*|}{2}+\sqrt{RE}\right)^2+r(1-R)E\right)\right) \;ff_{*} E^{\frac{\gamma}{2}}\sqrt{R}(1-R) \\
  &\hspace{3cm}\,dv dv_{*} dI dI_{*} dr dR\\
     &=\int \left\{\frac{8\pi}{\gamma+1}\frac{1}{(RE)^{\frac{1-\gamma}{2}}}\right\} \exp\left(a\left(\frac{1}{2}\left(\frac{|v+v_*|}{2}+\sqrt{RE}\right)^2+r(1-R)E\right)\right) ff_{*} E^{\frac{\gamma}{2}}\sqrt{R}(1-R) \\
  &\hspace{3cm} \,dv dv_{*} dI dI_{*} dr dR   \\
    &= \frac{8\pi}{\gamma+1} \int \exp\left(a\left(\frac{1}{2}\left(\frac{|v+v_*|}{2}+\sqrt{RE}\right)^2+r(1-R)E\right)\right) \;ff_{*} \frac{E^{\frac{\gamma}{2}}}{E^{\frac{1-\gamma}{2}}}\frac{(1-R)}{R^{\frac{1-\gamma}{2}}} \\
  &\hspace{3cm}\,dv dv_{*} dI dI_{*} dr dR.\\   
   \end{aligned}
   \end{equation*}
 Since $E\geq 1$, we have $\frac{1}{E^{\frac{3}{2}-\gamma}}\leq 1$. Therefore, carrying out the integration in $r$ first, we obtain
 \begin{equation*}
 \begin{aligned}
     &\|Q^{+}(f,f)\|_{1-\gamma} \\
     &\leq \frac{8\pi}{a(1+\gamma)} \int \exp\left(a\left(\frac{1}{2}\left(\frac{|v+v_*|}{2}+\sqrt{RE}\right)^2+(1-R)E\right)\right) \;ff_{*} \frac{1}{(1-R)E}\textcolor{black}{\frac{1}{E^{\frac{1}{2}-\gamma}}}\frac{(1-R)}{R^{\frac{1-\gamma}{2}}}\, dv dv_{*} dI dI_{*}  dR \\
    \hspace{-3cm}&= \frac{8\pi}{a(1+\gamma)} \int \exp\left(a\left(\frac{1}{2}\left(\frac{|v+v_*|}{2}+\sqrt{RE}\right)^2+(1-R)E\right)\right) \;ff_{*} \textcolor{black}{\frac{1}{E^{\frac{3}{2}-\gamma}}}\frac{1}{R^{\frac{1-\gamma}{2}}}\, dv dv_{*} dI dI_{*}  dR \\
    &\leq \frac{8\pi}{a(1+\gamma)} \int \exp\left(a\left(\left(\frac{|v+v_*|^2}{4}+ RE\right) +(1-R)E\right)\right) \;ff_{*} \textcolor{black}{\frac{1}{E^{{\frac{3}{2}-\gamma}}}} \frac{1}{R^{\frac{1-\gamma}{2}}}\,dRdv dv_{*} dI dI_{*}  \\
    &\leq \frac{8\pi}{a(1+\gamma)} \int \exp\left(a\left(\frac{|v+v_*|^2}{4}+{RE}+(1-R)E\right)\right) \;ff_{*}\frac{1}{R^{\frac{1-\gamma}{2}}} \, dRdv dv_{*} dI dI_{*} \\
    &= \frac{8\pi}{a(\gamma+1)} \int \exp\left(a\left(\frac{|v+v_*|^2}{4}+{E}\right)\right) \;ff_{*} \left\{
    \int^1_0 \frac{1}{R^{\frac{1-\gamma}{2}}}dR
    \right\}dv dv_{*} dI dI_{*}\\
    &= \frac{16\pi}{a(\gamma+1)^2} \int \varphi(v,I)\varphi(v_*,I_*) \;ff_{*} \,dv dv_{*} dI dI_{*}\\
    &\leq \frac{16\pi}{a(\gamma+1)^2} \|f\|_0^2.
   \end{aligned}
   \end{equation*}
\end{proof}

\section{Invariance of the solution map $\Psi$}
\label{invariance}

In this section, we will provide the necessary proofs for the existence. We will first prove invariance, and then establish contraction, both of which are required to apply the Banach fixed point theorem.

For this, we define our solution space $\mathcal{A}$ by
   \begin{align*}
        \mathcal{A}=\left\{ f \in L^1(\mathbb{R}^3_v\times \mathbb{R}^+_I;L^{\infty}_x(0,1)): f \geq 0,\|f\|_0 \leq a_{1}, L(f) \geq a_{2}\left(1+\left(|v|^2+I\right)^{\frac{\gamma}{2}}\right),\|f\|_{1-\gamma} \leq a_{3},\|f\|_{P} \leq a_{4}\right\}.
    \end{align*}
 The value of $a_i$ $(i=1,2,3,4)$ will be clarified in the proof.   
We now define the solution map $\Psi:\mathcal{A}\rightarrow \mathcal{A}$ as follows:
\[
\Psi(f)=\Psi^+(f)\mathbbm{1}_{v_1>0}
+\Psi^-(f)\mathbbm{1}_{v_1<0},
\]
where
 \begin{align}\label{solution map}
 \begin{split}
     \Psi^+(f)(x,v,I)&=\exp \left(-\frac{\varepsilon}{ |v_{1}|} \int_{0}^{x} L(f)(y, v,I) \,d y\right)f_L(v,I)\\
  &+\frac{\varepsilon}{ |v_{1}|} \int_{0}^{x}\exp \left(-\frac{\varepsilon}{ |v_{1}|} \int_{y}^{x} L(f)(z, v,I) \,d z\right)Q^+(f,f)\,dy\qquad (v_1>0)\\
    \Psi^-(f)(x,v,I)&=\exp \left(-\frac{\varepsilon}{ |v_{1}|} \int_{x}^{1} L(f)(y, v,I)\, d y\right)f_R(v,I)\\
  &+\frac{\varepsilon}{ |v_{1}|} \int_{x}^{1}\exp \left(-\frac{\varepsilon}{ |v_{1}|} \int_{x}^{y} L(f)(z, v,I) \,d z\right)Q^+(f,f)\,dy\qquad (v_1<0).
  \end{split}
	\end{align}

In this section, we  prove that the solution map $\Psi$ is invariant in $\mathcal{A}$. 

\begin{proposition} Let $1/2\leq \gamma\leq 1$. Assume $f_{LR}$ satisfies \eqref{ibc1}. Then $\Psi$ is invariant in $\mathcal{A}$. That is, $\Psi(f)\in \mathcal{A}$ for $f\in\mathcal{A}$.    
\end{proposition}
To prove this proposition, we need to show that
 $$\Psi(f) \geq 0,\; \|\Psi(f)\|_0 \leq a_1,\; L(\Psi(f)) \geq a_2\left(1+\left(|v|^2+I\right)^{\frac{\gamma}{2}}\right), \;\|\Psi(f)\|_{1-\gamma} \leq a_3,\; \text{and \;}\|\Psi(f)\|_{P} \leq a_4$$
 for some positive constants $a_j$ $(j=1,2,3,4)$,
 which will be given in the following series of lemmas.
Throughout the proof, we will only consider $\Psi^{+}(f)$, since the proof for $\Psi^{-}(f)$ is almost identical.
 \begin{lemma} $\Psi$ preserves non-negativity:
  $$\Psi(f) \geq 0$$
  for $f\in\mathcal{A}.$
 \end{lemma}
\begin{proof}

This follows directly from the definition of the solution map:
\begin{align*}
    \Psi^+(f) &= \exp\left(-\frac{\varepsilon}{|v_1|}\int_0^x L(f)\,dy\right)f(0,v,I) + \frac{\varepsilon}{|v_1|}\int_0^x \exp\left(-\frac{\varepsilon}{|v_1|}\int_y^x L(f) \, dz \right) Q^+(f,f) \, dy \\
        &\geq \exp\left(-\frac{\varepsilon}{|v_1|}\int_0^x L(f)\,dy\right)f_L \geq 0.
\end{align*}
\end{proof}

\begin{lemma}\label{invariance 1} Define 
 $a_1 \equiv 2|||f_{LR}|||$.
Then, for sufficiently small $\varepsilon$,  we have
 $\|\Psi(f)\|_0 \leq a_1$
 whenever $f\in\mathcal{A}$.
\end{lemma}
\begin{proof}
     From the definition of $\Psi^+(f)$, we have
   \begin{equation*}
    \begin{aligned}
        \|\Psi^+(f)\|_0 
        &= \int \varphi(v,I) \|\Psi^+(f)\|_{L_x^{\infty}} \, dvdI \\
        &\leq \int \varphi(v,I) \left\|\exp\left(-\frac{\varepsilon}{|v_1|}\int_0^x L(f)\,dy\right)f_{L}(v,I) \right\|_{L_x^{\infty}} \, dvdI \\
        &+ \int \varphi(v,I) \left\| \frac{\varepsilon}{|v_1|}\int_0^x \exp\left(-\frac{\varepsilon}{|v_1|}\int_y^x L(f) \, dz \right) Q^+(f,f) \, dy \right\|_{L_x^{\infty}} \, dvdI\\
        &\equiv I+II.
    \end{aligned}
    \end{equation*}
To estimate $I$, we recall $L(f)\geq a_2$ to get
    \begin{align*}
        I&=\int \varphi(v,I) \left\|\exp\left(-\frac{\varepsilon}{|v_1|}\int_0^x L(f)\,dy\right)f_L(v,I) \right\|_{L_x^{\infty}} \, dvdI \\
        &\leq \int \varphi(v,I) \left\|\exp\left(-\frac{\varepsilon}{|v_1|}x a_2\right) f_L \right\|_{L_x^{\infty}} \, dvdI \\
        &\leq\frac{a_1}{2}.
    \end{align*}

    Then, the first part is bounded by $\frac{a_1}{2}$. 
    For the second part, we have \\
   \begin{equation*}
   \begin{aligned}
        II&=\int \varphi(v,I)\left \|\int_{0}^{x}\frac{\varepsilon}{|v_1|} \exp\left(-\frac{\varepsilon}{|v_1|} \int_y^x L(f)\, dz\right) Q^{+}(f,f) \,dy \right\|_{L_x^{\infty}} \,dv dI \\
        &\leq \int \varphi(v,I)\left \|\int_{0}^x\frac{\varepsilon}{|v_1|} \exp\left(-\frac{\varepsilon}{|v_1|} \int_y^x \!\!a_2\,dz\right) Q^{+}(f,f) \,dy \right\|_{L_x^{\infty}} \,dv dI
        \\
        &\leq \int \varphi(v,I) \frac{\varepsilon}{|v_1|}\left(\int_0^1 \exp\left(- a_2\frac{\varepsilon}{|v_1|}\tau \right)\, d\tau\right) \|Q^{+}(f,f)\|_{L_x^{\infty}}\, dv dI,
    \end{aligned}
    \end{equation*}
    where $\tau=x-y$. 
    We then divide this integral into small $v_1$ region and  large $v_1$ region:  \\
    \begin{align*}
        &\int_{\mathbb{R}^2 \times \mathbb{R}^{+}} \int_{|v_1| \leq \frac{1}{\varepsilon}}  \varphi(v,I) \frac{\varepsilon}{|v_1|} \left(\int_0^1 \exp\left(- a_2\frac{\varepsilon}{|v_1|}\tau \right)\, d\tau\right) \|Q^{+}(f,f) \|_{L_x^{\infty}}\, dvdI \\
        &+ \int_{\mathbb{R}^2 \times \mathbb{R}^{+}} \int_{|v_1|>\frac{1}{\varepsilon}} \varphi(v,I) \frac{\varepsilon}{|v_1|} \left(\int_0^1 \exp\left(- a_2\frac{\varepsilon}{|v_1|}\tau \right)\, d\tau\right) \|Q^{+}(f,f) \|_{L_x^{\infty}} \,dvdI\\
        &\equiv II_1+II_2.
    \end{align*}
For the estimate of $II_1$, we compute \\
\begin{equation}\label{app}
    \begin{aligned}
         &\int_{\left|v_{1}\right|<\frac{1}{\varepsilon}}\left(1-\exp\left(- a_2\frac{\varepsilon}{|v_{1}|}\right)\right)\,d v_{1} \\
         &= 2\int_{0}^{\varepsilon}\left(1-\exp\left(- a_2\frac{\varepsilon}{v_{1}}\right)\right)\,d v_{1}+ 2\int_{\varepsilon}^{\frac{1}{\varepsilon}}\left(1-\exp\left(- a_2\frac{\varepsilon}{v_{1}}\right)\right)\,d v_{1}\\
         &\leq 2\varepsilon + 2a_2 \int_{\varepsilon}^{\frac{1}{\varepsilon}} \frac{\varepsilon}{v_1}\, dv_1\\
         & \leq 2\varepsilon + 2 a_2\varepsilon\ln{\frac{1}{\varepsilon}} - 2a_2{\varepsilon}\ln \varepsilon \\
         &= 2\varepsilon+4a_2\varepsilon\ln{\frac{1}{\varepsilon}}.
    \end{aligned}
\end{equation}
to obtain
\begin{equation*}
\begin{aligned}
    II_1&=\int_{\mathbb{R}^2 \times \mathbb{R}^{+}} \int_{|v_1| \leq \frac{1}{\varepsilon}} \frac{\varepsilon}{|v_1|}\varphi(v,I) \left(\int_0^1 \exp\left(- a_2\frac{\varepsilon}{|v_1|}\tau \right)\, d\tau\right)\| Q^{+}(f,f)  \|_{L_x^{\infty}} \,dvdI \\
    &\leq\frac{1}{a_2} \int_{|v_1| \leq \frac{1}{\varepsilon}} \left(  1-\exp\left(- a_2\frac{\varepsilon}{|v_1|}\right) \right) \left( \int_{\mathbb{R}^2 \times \mathbb{R}^{+}} \varphi(v,I) \|Q^{+}(f,f)\|_{L_x^{\infty}} \,dv_2 dv_3 dI \right) \,dv_1 \\
    &\leq\frac{1}{a_2} \|Q^{+}(f,f)\|_{P} \int_{|v_1| \leq \frac{1}{\varepsilon}} \left(  1-\exp\left(- a_2\frac{\varepsilon}{|v_1|}\right) \right) \, dv_1 \\
    &\leq \frac{1}{a_2}\left(2\varepsilon+4a_2\varepsilon\ln{\frac{1}{\varepsilon}} \right)\|Q^{+}(f,f)\|_{P}. 
\end{aligned}
\end{equation*}

For $II_2$, we use  $\int_0^1 \exp\left(-\frac{\varepsilon}{|v_1|}\tau\right)d\tau\leq 1$ to get
\begin{align*}
    II_2&= \int_{\mathbb{R}^2 \times \mathbb{R}^{+}} \int_{|v_1| > \frac{1}{\varepsilon}} \frac{\varepsilon}{|v_1|}\varphi(v,I) \left(\int_0^1 \exp\left(-\frac{\varepsilon}{|v_1|}\tau \right) \,d\tau\right)\| Q^{+}(f,f)  \|_{L_x^{\infty}} \,dvdI\\
    &\leq \int_{\mathbb{R}^2 \times \mathbb{R}^{+}} \int_{|v_1|>\frac{1}{\varepsilon}} \frac{\varepsilon}{|v_1|}\varphi(v,I) \|Q^{+}(f,f)\|_{L_x^{\infty}}   \,dvdI  \\
    &\leq \varepsilon^2 \int_{\mathbb{R}^3 \times \mathbb{R}^+} \varphi(v,I) \|Q^{+}(f,f)\|_{L_x^{\infty}} \,dv dI \\
    &= \varepsilon^2 \|Q^{+}(f,f)\|_0.
\end{align*}
We combine the estimate of $I$ and $II$, and recall the estimates on $Q^+$ in Lemma \ref{nbb}
and \ref{planenorm} to derive
\begin{equation*}
\begin{aligned}
    \|\Psi^{+}(f)\|_0
    &\leq \frac{a_1}{2} +  \left(\frac{1}{a_2} \left(2\varepsilon+4a_2\varepsilon\ln{\frac{1}{\varepsilon}}\right)  \|Q^{+}(f,f)\|_{P}+ \varepsilon^2\|Q^{+}(f,f)\|_0 \right) \\
    &\leq \frac{a_1}{2} +  \left( \frac{1}{a_2}\left(2\varepsilon+4a_2\varepsilon\ln{\frac{1}{\varepsilon}}\right)\max\{\pi,2^{2-\gamma}\} (\|f\|_0\cdot {\|f\|_{1-\gamma}}+\|f\|^2_0)+ 4\pi \max\left\{\frac{1}{a},1\right\} \varepsilon^2\|f\|_0^2 \right).
\end{aligned}
\end{equation*}
Since $f \in \mathcal{A}$, $f$ satisfies $\|f\|_0 \leq a_1 $ and $\|f\|_{1-\gamma} \leq a_3$. Thus, for sufficiently small $\varepsilon$, we have 

$$\|\Psi^+(f)\|_0\leq \frac{a_1}{2}+\frac{1}{a_2}  \left(4a_2\varepsilon \ln\frac{1}{\varepsilon}+2\varepsilon \right) \max\{\pi,2^{2-\gamma}\} (a_1\cdot a_3+ a_1^2) + 4\pi \max\left\{\frac{1}{a},1\right\} \varepsilon^2 a_1^2 \leq a_1.$$ \\
\end{proof}

\begin{lemma}\label{invariance L lower bound}
For $a_1$ defined in Lemma \ref{invariance 1} and 
\begin{align*}
   c_\alpha &=\int_{[0,1]^2\times \mathbb{S}^2}  (r(1-r))^{\alpha}(1-R)^{2\alpha+1} R^{\frac{1}{2}} \,dR dr d\sigma,\\
C_1&=\int_{|v_*|\leq1}\exp \left(-4\pi \frac{1}{ |v_{*1}|} \frac{e^a}{a}\left(1+\left(|v_*|^2+I_*\right)^{\frac{\gamma}{2}}\right) a_1\right) f_L(v_*,I_*)\, dv_*dI_*, \\
and\;\,C_2&=\int_{|v_*|\geq4}\exp \left(-4\pi \frac{1}{ |v_{*1}|} \frac{e^a}{a}\left(1+\left(|v_*|^2+I_*\right)^{\frac{\gamma}{2}}\right) a_1\right) f_L(v_*,I_*)\, dv_*dI_*,
    \end{align*}
define 
    \begin{align*}
        a_2\equiv c_{\alpha} \min\left\{ \frac{1}{4^{\gamma}} C_1, C_2\right\}>0.
    \end{align*}
Then, for $f\in\mathcal{A}$, we have
    \begin{align*}
        L(\Psi(f))\geq a_2\left(1+\left(|v|^2+I\right)^{\frac{\gamma}{2}}\right).
    \end{align*}
\end{lemma}

\begin{proof}
    From the definition of $c_{\alpha}$, we write
  \begin{equation*}
   \begin{aligned}
     L(\Psi^+(f))&=    \int_{A}\Psi^+(f)(v_*,I_*)E^{\frac{\gamma}{2}}  (r(1-r))^{\alpha}(1-R)^{2\alpha+1} R^{\frac{1}{2}} \,dR dr d\sigma dI_* dv_*\\
     &= c_{\alpha}\int_{\mathbb{R}^3\times\mathbb{R}^+}\Psi^+(f)(v_*,I_*)E^{\frac{\gamma}{2}}   \,dv_* dI_*.
    \end{aligned}
    \end{equation*}

To compute $\int \Psi^+(f)\,dv_*dI_*$, we use the definition of $\Psi^+(f)$, which gives
     \begin{equation*}
   \begin{aligned}
     \int\Psi^+(f)(v_*,I_*)\,dv_*dI_*
     &\geq  \int\exp \left(-\frac{\varepsilon}{ |v_{*1}|} \int_{0}^{x} L(f)(y, v_*,I_*) \,d y\right)f_L(v_*,I_*)\,dv_*dI_*.
     \end{aligned}
    \end{equation*}
 Recalling Lemma \ref{lemma6rs}, we proceed further as
     \begin{equation*}
   \begin{aligned}
        \int \Psi^+(f)(v_*,I_*)\,dv_*dI_*
        &\geq  \int\exp \left(-4\pi \frac{\varepsilon}{ |v_{*1}|} \frac{e^a}{a} \left(1+\left(|v_*|^2+I_*\right)^{\frac{\gamma}{2}}\right) \|f\|_0 \right)f_L(v_*,I_*)\,dv_*dI_*\\
        &\geq \int\exp \left(-4\pi \frac{\varepsilon}{ |v_{*1}|} \frac{e^a}{a}\left(1+\left(|v_*|^2+I_*\right)^{\frac{\gamma}{2}}\right) a_1\right) f_L(v_*,I_*)\,dv_*dI_*\\
            &\geq \int\exp \left(-4\pi \frac{1}{ |v_{*1}|} \frac{e^a}{a}\left(1+\left(|v_*|^2+I_*\right)^{\frac{\gamma}{2}}\right) a_1\right) f_L(v_*,I_*)\, dv_*dI_*\\
            &\equiv C_0 >0.
        \end{aligned}
    \end{equation*}
  In the last line, we assumed without loss of generality that $\varepsilon<1$. First, for the case $|v|\geq2$, we have
\begin{align*}
    \int_{\mathbb{R}^3\times\mathbb{R}^+}\Psi^+(f)(v_*,I_*)E^{\frac{\gamma}{2}} \,dv_* dI_* &= \int_{\mathbb{R}^3\times\mathbb{R}^+}\Psi^+(f)(v_*,I_*)\left( \frac{1}{4}|v-v_*|^2+I+I_*\right)^{\frac{\gamma}{2}} \,dv_* dI_* \\
    &\geq \int_{|v_*|\leq 1} \Psi^+(f)\left(\frac{1}{4}|v-v_*|^2+I+I_*\right)^{\frac{\gamma}{2}}\, dv_*dI_* \\
    &\geq \int_{|v_*|\leq 1} \Psi^+(f)\left(\frac{1}{4}||v|-|v_*||^2+I+I_*\right)^{\frac{\gamma}{2}}\, dv_*dI_* \\
    &\geq \int_{|v_*|\leq 1} \Psi^+(f)\left(\frac{1}{4}\left||v|-1\right|^2+I\right)^{\frac{\gamma}{2}}\, dv_*dI_* \\
    &\geq \int_{|v_*|\leq 1} \Psi^+(f)\left(\frac{1}{16}|v|^2+I\right)^{\frac{\gamma}{2}}\, dv_*dI_*.
\end{align*}
In the last line, we used $2|v|-2\geq |v|$. Thus, we can get
\begin{equation}\label{invariance L 1}
    \begin{aligned}
        &\int_{\mathbb{R}^3\times\mathbb{R}^+}\Psi^+(f)(v_*,I_*)E^{\frac{\gamma}{2}} \,dv_* dI_* \\
        &\geq 
 \frac{1}{4^{\gamma}} \left(|v|^2+I\right)^{\frac{\gamma}{2}} \int_{|v_*|\leq 1} \Psi^+(f)\, dv_*dI_* \\
 &\geq \frac{1}{4^{\gamma}} \left(|v|^2+I\right)^{\frac{\gamma}{2}} \int_{|v_*|\leq1}\exp \left(-4\pi \frac{1}{ |v_{*1}|} \frac{e^a}{a}\left(1+\left(|v_*|^2+I_*\right)^{\frac{\gamma}{2}}\right) a_1\right) f_L(v_*,I_*)\, dv_*dI_*\\
 &\geq \frac{1}{4^{\gamma}} C_1\left(|v|^2+I\right)^{\frac{\gamma}{2}} >0.
    \end{aligned}
\end{equation}

Next, for the case $|v|<2$, we can obtain
\begin{equation*}
    \begin{aligned}
        \int_{\mathbb{R}^3\times\mathbb{R}^+}\Psi^+(f)(v_*,I_*)E^{\frac{\gamma}{2}} \,dv_* dI_* &\geq \int_{|v_*|\geq 4} \Psi^+(f)\left(\frac{1}{4}||v_*|-|v||^2+I+I_*\right)^{\frac{\gamma}{2}}\, dv_*dI_* \\
    &\geq \int_{|v_*|\geq 4} \Psi^+(f) (1+I+I_*)^{\frac{\gamma}{2}}\, dv_*dI_* \\
    &\geq \int_{|v_*|\geq 4} \Psi^+(f) \, dv_*dI_*.
    \end{aligned}
    \end{equation*}
    Similar to the previous case, by the definition of $\Psi^+$,
    \begin{equation}\label{invariance L 2}
        \begin{aligned}
    &\int_{\mathbb{R}^3\times\mathbb{R}^+}\Psi^+(f)(v_*,I_*)E^{\frac{\gamma}{2}} \,dv_* dI_* \\
    &\geq \int_{|v_*|\geq4}\exp \left(-4\pi \frac{1}{ |v_{*1}|} \frac{e^a}{a}\left(1+\left(|v_*|^2+I_*\right)^{\frac{\gamma}{2}}\right) a_1\right) f_L(v_*,I_*)\, dv_*dI_*\\
    &\geq C_2>0.
    \end{aligned}
\end{equation}

Combining the estimates \eqref{invariance L 1} and \eqref{invariance L 2} to get
\begin{align*}
    L(\Psi^+(f)) &\geq c_{\alpha}\left\{ \frac{1}{4^{\gamma}} C_1\left(|v|^2+I\right)^{\frac{\gamma}{2}}+ C_2\right\} \\
    &\geq c_{\alpha} \min\left\{ \frac{1}{4^{\gamma}} C_1, C_2\right\} \left(1+\left(|v|^2+I\right)^{\frac{\gamma}{2}}\right).
\end{align*}
By the definition of $a_2$, we can get the desired result.
\end{proof}

\begin{lemma}\label{invariance gamma} Define
\begin{align*}
    a_3\equiv\frac{a_1}{2}+\frac{1}{a_2}\frac{16\pi}{(1+\gamma)^2}\max\left\{\frac{1}{a}, 1\right\}a_1^2.
\end{align*} 
Then we have
    $$\|\Psi(f)\|_{1-\gamma} \leq a_3$$
    for $f\in\mathcal{A}$.
\end{lemma}
\begin{proof}
From the definition of $\Psi^+(f)$,
    \begin{align*}
         \|\Psi^{+}(f)\|_{1-\gamma} 
        &= \sup_w \int \frac{1}{|v+w|^{1-\gamma}} \varphi(v,I) \|\Psi^{+}(f)\|_{L_x^{\infty}} \,dv dI \\
        &\leq \sup_w \int \frac{1}{|v+w|^{1-\gamma}} \varphi(v,I) \left\|\exp \left(-\frac{\varepsilon}{|v_1|} \int_0^x L(f) \,dy \right) f_{L}\right\|_{L_x^{\infty}} \,dv dI \\
        &+ \sup_w \int \frac{1}{|v+w|^{1-\gamma}} \varphi(v,I) \left\|\frac{\varepsilon}{|v_1|} \int_0^x \exp \left( -\frac{\varepsilon}{|v_1|} \int_y^x L(f)\, dz \right) Q^{+}(f,f)\,dy\right\|_{L_x^{\infty}}\, dv dI \\
        &\equiv I+II.
    \end{align*}
The estimate of $I$ is straightforward:
   \begin{align*}
         I&=\sup_w \int \frac{1}{|v+w|^{1-\gamma}} \varphi(v,I) \left\|\exp \left(-\frac{\varepsilon}{|v_1|} \int_0^x L(f) \,dy \right) f_{L}\right\|_{L_x^{\infty}} \,dv dI\\
         &=\sup_w \int \frac{1}{|v+w|^{1-\gamma}} \varphi(v,I) \|f_{L}\|_{L_x^{\infty}} \,dv dI\\
         &=\|f_L\|_{1-\gamma}\\
        &\leq\frac{a_1}{2}. 
    \end{align*}

For the second part, 
Since $f \in \mathcal{A}$, then $L(f) \geq a_2.$  Therefore,
$$\exp \left(-\frac{\varepsilon}{|v_1|} \int_y^x L(f)\, dz \right) \leq \exp \left( -\frac{\varepsilon}{|v_1|}a_2(x-y)\right).$$
Then by explicit calculation, the second part is bounded by
\begin{align*}
    \sup_w &\int \frac{1}{|v+w|^{1-\gamma}} \varphi(v,I) \left\|\frac{\varepsilon}{|v_1|} \int_0^x \exp \left( -\frac{\varepsilon}{|v_1|} \int_y^x L(f) \,dz \right) Q^{+}(f,f)\,dy\right\|_{L_x^{\infty}} \,dv dI \\
    &\leq \sup_w \int \frac{1}{|v+w|^{1-\gamma}} \varphi(v,I) \left\|\frac{\varepsilon}{|v_1|} \int_0^x \exp \left( -\frac{\varepsilon}{|v_1|} a_2(x-y) \right) Q^{+}(f,f)\,dy\right\|_{L_x^{\infty}} \,dv dI \\
    &\leq \sup_w \int \frac{1}{|v+w|^{1-\gamma}} \varphi(v,I) \frac{1}{a_2}\left(1-\exp\left(-\frac{\varepsilon}{|v_1|}a_2\right)\right) \left\| Q^{+}(f,f) \right\|_{L_x^{\infty}} \,dv dI \\
    &\leq \frac{1}{a_2} \sup_w \int \frac{1}{|v+w|^{1-\gamma}} \varphi(v,I) \|Q^{+}(f,f)\|_{L_x^{\infty}} \,dv dI \\
    &= \frac{1}{a_2}\|Q^{+}(f,f)\|_{1-\gamma} .
\end{align*}
Recalling Lemma \ref{singular norm}, we have for$1/2 \leq \gamma \leq 1$,
$$\frac{1}{a_2}\|Q^{+}(f,f)\|_{1-\gamma} 
    \leq \frac{1}{a_2}\frac{16\pi}{(1+\gamma)^2}\max\left\{1, \frac{1}{a}\right\} \|f\|_0^2 \leq \frac{1}{a_2}\frac{16\pi}{(1+\gamma)^2}\max\left\{1, \frac{1}{a}\right\} a_1^2.$$
Therefore,    
\[
\|\Psi^+(f)\|_{1-\gamma}\leq \frac{a_1}{2}+
\frac{1}{a_2}\frac{16\pi}{(1+\gamma)^2}\max\left\{1, \frac{1}{a}\right\} a_1^2.\]
Therefore, we set $$a_3=\frac{a_1}{2}+\frac{1}{a_2}\frac{16\pi}{(1+\gamma)^2}\max\left\{\frac{1}{a}, 1\right\}a_1^2 $$ to get the desired result.
\end{proof}

\begin{lemma} Define 
\[
a_4 \equiv \frac{a_1}{2}+ \max\{\pi,2^{2-\gamma}\} \frac{1}{a_2}\left(a_1^2 + a_1 \cdot a_3 \right).
\] Then, $\Psi(f)$ satisfies
    $$\|\Psi(f)\|_{P} \leq a_4$$
    whenever $f\in \mathcal{A}$.
\end{lemma}
\begin{proof}
From the definition of $\Psi^+(f)$ and the hyperplane norm $\|\cdot\|_P$,
\begin{align*}
    &\|\Psi^{+}(f)\|_{P} \\
    &= \sup_P \int_{P\times \mathbb{R}^+} \varphi(v,I) \left\|\Psi^{+}(f)\right\|_{L_x^{\infty}} \,d\pi_{P,v} dI \\
       &\leq \sup_P \int_{P\times \mathbb{R}^+} \varphi(v,I) \left\|\exp \left(-\frac{\varepsilon}{|v_1|}\int_0^x L(f) \,dy\right) f(0,v,I)\right\|_{L_x^{\infty}}\, d\pi_{P,v} dI \\
    &+ \sup_P \int_{P\times \mathbb{R}^+} \varphi(v,I) \left\|\frac{\varepsilon}{|v_1|} \int_0^x \exp \left( -\frac{\varepsilon}{|v_1|} \int_y^x L(f) \,dz \right) Q^{+}(f,f) \,dy \right\|_{L_x^{\infty}} \,d\pi_{P,v} dI\\
    &\equiv I+II.
\end{align*}
The estimate for $I$ is straightforward:
\begin{align*}
    I&=\sup_P \int_{P\times \mathbb{R}^+} \varphi(v,I) \left\|\exp \left(-\frac{\varepsilon}{|v_1|}\int_0^x L(f) \,dy\right) f_L(v,I)\right\|_{L_x^{\infty}}\, d\pi_{P,v} dI \\
    &\leq \sup_P \int_{P\times \mathbb{R}^+} \varphi(v,I) \| f_L(v,I)\|_{L_x^{\infty}}\, d\pi_{P,v} dI \\
    &= \|f_L\|_P \\
    &\leq \frac{a_1}{2}.
\end{align*}

To estimate $II$, we use $L(f)\geq a_2$ to get
\begin{align*}
    & \sup_P \int_{P\times \mathbb{R}^+} \varphi(v,I) \left\|\frac{\varepsilon}{|v_1|} \int_0^x \exp \left( -\frac{\varepsilon}{|v_1|} \int_y^x L(f)\, dz \right) Q^{+}(f,f) \,dy\right\|_{L_x^{\infty}} \,d\pi_{P,v} dI \\
    &\leq \sup_P \int_{P\times \mathbb{R}^+} \varphi(v,I) \left\|\frac{\varepsilon}{|v_1|} \int_0^x \exp \left( -\frac{\varepsilon}{|v_1|} a_2(x-y) \right) Q^{+}(f,f) \,dy\right\|_{L_x^{\infty}} \,d\pi_{P,v} dI \\
    &\leq \sup_P \int_{P\times \mathbb{R}^+} \varphi(v,I) \frac{\varepsilon}{|v_1|} \left(\frac{|v_1|}{a_2 \varepsilon} \right) \|Q^+(f,f)\|_{L_x^{\infty}} \,d\pi_{P,v} dI. \\
\end{align*}
Therefore, \newline
\begin{align*}
    \|\Psi^{+}(f)\|_{P}&\leq \frac{a_1}{2} + \frac{1}{a_2} \sup_P \int_{P\times \mathbb{R}^+} \varphi(v,I) \|Q^{+}(f,f)\|_{L_x^{\infty}}\, d\pi_{P,v} dI \\
    &\leq \frac{a_1}{2} + \frac{1}{a_2}\|Q^{+}(f,f)\|_P \\
    &\leq \frac{a_1}{2} + \max\{\pi,2^{2-\gamma}\} \frac{1}{a_2} \left(\|f\|_0^2 + \|f\|_0 \cdot \|f\|_{1-\gamma} \right) \\
    &\leq \frac{a_1}{2} + \max\{\pi,2^{2-\gamma}\}\frac{1}{a_2}\left(a_1^2 + a_1 \cdot a_3 \right). 
\end{align*}
\noindent Thus, we can get the desired result by setting $a_4 = \frac{a_1}{2}+ \max\{\pi,2^{2-\gamma}\}\frac{1}{a_2}\left(a_1^2 + a_1 \cdot a_3 \right)$. \\
\end{proof}

\vspace{-1em}
\section{Contraction of the solution map \texorpdfstring{$\Psi$}{}}\label{contraction}
    We now turn to the contraction property of the solution map $\Psi$ in $\|\cdot\|_0$ under the suitable smallness condition given in Theorem \ref{main theorem}. 
    
    \begin{theorem}[Contraction property]
        Let $1/2\leq \gamma\leq 1$. Suppose $f_{LR}$ satisfies the assumptions in Theorem \ref{main theorem}. Let $f \in \mathcal{A}$.  Then we have  
        \begin{align*}
            \|\Psi(f)-\Psi(g)\|_0 \leq C(\varepsilon)\|f-g\|_0,
        \end{align*}
        where $C(\varepsilon)$ satisfies
        $\displaystyle\lim_{\varepsilon\rightarrow0}C(\varepsilon)=0$.
    \end{theorem}
    \begin{proof}
    We focus on the proof for $\Psi^+$  since the proof for $\Psi^-$ is almost identical. From the definition of $\Psi^+$, we have
        \begin{align*}
        &\| \Psi^+(f) - \Psi^+(g) \|_0 
        = \int \varphi(v,I) \| \Psi^+(f)-\Psi^+(g)\|_{L_x^{\infty}} \, dvdI \\
        &\leq \int \varphi(v,I) \left\| \exp\left(-\frac{\varepsilon}{|v_1|}\int_0^x L(f)\,dy\right)f_L - \exp\left(-\frac{\varepsilon}{|v_1|}\int_0^x L(g)\,dy\right)f_L \right\|_{L_x^{\infty}} \, dvdI \\
        &+ \int \varphi(v,I) \left\| \frac{\varepsilon}{|v_1|}\int_0^x
        e^{-\frac{\varepsilon \int_y^x L(f) \, dz }{|v_1|}} 
        Q^+(f,f) \, dy - \frac{\varepsilon}{|v_1|}\int_0^x 
        e^{-\frac{\varepsilon \int_y^x L(g) \, dz }{|v_1|}}
        Q^+(g,g) \, dy \right\|_{L_x^{\infty}} \, dvdI \\
        &\equiv I+II.
        \end{align*}
        
By the mean value theorem, there is $0<\theta<1$ so that
        \begin{align*}
        I&= \int \varphi(v,I) \left\| \left( \exp\left(-\frac{\varepsilon}{|v_1|}\int_0^x L(f)\,dy\right) - \exp\left(-\frac{\varepsilon}{|v_1|}\int_0^x L(g)\,dy\right) \right) f_L \right\|_{L_x^{\infty}} \, dvdI\\
        &\leq \int \varphi(v,I) \left\| \frac{\varepsilon}{|v_1|} \left(\int_0^x L(f)-L(g) \, dy \right) \exp\left(-\frac{\varepsilon}{|v_1|}\int_0^x \theta L(f) + (1-\theta) L(g) \, dy \right) \right\|_{L_x^{\infty}} \|f_L\|_{L_x^{\infty}} \, dvdI.
        \end{align*}
        We then use lemma \ref{lemma6rs} to estimate $I$ as follows:
        \begin{align*}
       I &\leq \int \varphi(v,I) \left\| \frac{\varepsilon}{|v_1|}x \left(\frac{4\pi e^a}{a}\right) \left(1+\left(|v|^2+I\right)^{\frac{\gamma}{2}}\right) \|f-g\|_0 \right. \\
        &\hspace{5cm}\left. \times\exp\left(-\frac{\varepsilon}{|v_1|}\int_0^x \theta L(f) + (1-\theta) L(g) \, dy \right) \right\|_{L_x^{\infty}} \|f_L\|_{L_x^{\infty}} \, dvdI \\
        &\leq \int \varphi(v,I) \left\| \frac{\varepsilon}{|v_1|} \left(\frac{4\pi e^a}{a}\right)x \left(1+\left(|v|^2+I\right)^{\frac{\gamma}{2}}\right) \|f-g\|_0
         \right\|_{L_x^{\infty}} \|f_L\|_{L_x^{\infty}} \, dvdI \\
        &\leq \left(\frac{4\pi e^a}{a}\right) \varepsilon \|f-g\|_0 \int \varphi(v,I) \frac{1}{|v_1|} \left(1+\left(|v|^2+I\right)^{\frac{\gamma}{2}}\right) \|f_L\|_{L_x^{\infty}} \,dvdI \\
        &\leq \left(\frac{4\pi e^a}{a}\right)\varepsilon \left\|\frac{1}{|v_1|}\left(1+\left(|v|^2+I\right)^{\frac{\gamma}{2}}\right)f_L\right\|_0 \|f-g\|_0 \\
        &\leq C \varepsilon \|f-g\|_0 \\
        \end{align*}

We now turn to the estimate of $II$. First, we split $II$
\begin{equation*}
    \begin{aligned}
    II&= \left\| \frac{ \varepsilon}{|v_1|}\left(\int_{0}^{x}\exp\left(-\frac{\varepsilon}{|v_1|}\int_{y}^{x}L(f)(z,v,I)\,dz\right)Q^+(f,f)-\exp\left(-\frac{\varepsilon}{|v_1|}\int_{y}^{x}L(g)(z,v,I)\,dz\right)Q^+(g,g)\,dy\right)\right\|_0\\
    \end{aligned}
    \end{equation*}
    \begin{equation}\label{ln1}
        \begin{aligned}
&\leq \int_{\mathbb{R}^{3}\times \mathbb{R}^+} \varphi(v,I)\frac{ \varepsilon}{|v_1|}\times\bigg(\left\|\int_{0}^{x}\left(Q^{+}(f, f)-Q^{+}(g, g)\right)\exp\left(-\frac{\varepsilon}{|v_1|}\int_{y}^{x}L(f)(z,v,I)\,dz\right)\,dy \right\|_{L_x^{\infty}}\\
  & + \left\|\int_{0}^{x}\left\{\exp\left(-\frac{\varepsilon}{|v_1|}\int_{y}^{x}L(f)(z,v,I)\,dz\right)- \exp\left(-\frac{\varepsilon}{|v_1|}\int_{y}^{x}L(g)(z,v,I)\,dz\right)\right\}Q^{+}(g, g)\,d y\right\|_{L_x^{\infty}}  \bigg) \,dvdI \\
  &\equiv II_1+II_2.
  \end{aligned}
  \end{equation}
For $II_1$, We observe
$$ 
\begin{aligned}
\left\|Q^{+}(f, f)-Q^{+}(g, g)\right\|_{L_x^{\infty}} 
&=\left\|Q^{+}(f-g, f+g)\right\|_{L_x^{\infty}}, \\
\end{aligned}$$  
and employ the lower bound of $L(f)$ and $L(g)$ given in Lemma \ref{invariance L lower bound} to find
\begin{equation*}
\begin{aligned}
   II_1&\leq\int_{\mathbb{R}^{3}\times \mathbb{R}^+} \varphi(v,I)\frac{\varepsilon}{\left|v_{1}\right| }\bigg(\left\|\int_{0}^{x}\left(Q^{+}(f, f)-Q^{+}(g, g)\right)\exp\left(-\frac{\varepsilon}{|v_1|}a_2({x}-{y})\right)\,d y\right\|_{L_x^{\infty}} \bigg) \,d vdI.
\end{aligned}
\end{equation*}
We set $\tau=a_2(x-y)$ and use 
$$
\int_{0}^{a_2} \exp\left(-\frac{\varepsilon}{|v_{1}|} \tau\right) d \tau=\frac{\left|v_{1}\right|}{\varepsilon}\left(1-\exp\left(-a_2\frac{\varepsilon}{|v_{1}|}\right)\right)
$$
to get
\begin{equation}\label{CC}
\begin{aligned}
   II_1
   &\leq \int_{\mathbb{R}_{v}^{3}\times\mathbb{R}^+} \varphi(v,I)\frac{\varepsilon}{|v_1|}\left\|Q^{+}(f-g, f+g)\right\|_{L_x^{\infty}}{\left(\frac{1}{a_2}\int_{0}^{a_2} \exp\left(-\frac{\varepsilon}{|v_{1}|} \tau\right)\, d \tau\right)} \,d vdI\\
  &\leq \int_{\mathbb{R}_{v}^{3}\times\mathbb{R}^+} \varphi(v,I)\frac{\varepsilon}{|v_1|}\left\|Q^{+}(f-g, f+g)\right\|_{L_x^{\infty}}
  \left\{\frac{1}{a_2}\frac{\left|v_{1}\right|}{\varepsilon}\left(1-\exp\left(-a_2\frac{\varepsilon}{|v_{1}|}\right)\right)  \right\}\,d vdI\\
    &\leq \frac{1}{a_2}\int_{\mathbb{R}_{v}^{3}\times\mathbb{R}^+} \varphi(v,I)\left\|Q^{+}(f-g, f+g)\right\|_{L_x^{\infty}}
  \left(1-\exp\left(-a_2\frac{\varepsilon}{|v_{1}|}\right)\right)  \,d vdI\\
\end{aligned}
\end{equation}
We split the last integral in \eqref{CC} into the large and small velocity parts: 
$$
\begin{aligned}
 II_2&\leq \frac{1}{a_2} \int_{\left|v_{1}\right|\leq \frac{1}{\varepsilon}}\left(1-\exp\left(-a_2\frac{\varepsilon}{|v_{1}|}\right)\right) \times
\left[  \int_{\mathbb{R}} \int_{\mathbb{R}} \int_{\mathbb{R}^+}\varphi(v,I)\left\|Q^{+}(f-g, f+g)\right\|_{L_x^{\infty}}\, d v_{2} d v_{3}dI\right]\, d v_{1}\\
& +   \frac{1}{a_2} \int_{\left|v_{1}\right|\geq \frac{1}{\varepsilon}}\left(1-\exp\left(-a_2\frac{\varepsilon}{|v_{1}|}\right)\right) \times
\left[  \int_{\mathbb{R}} \int_{\mathbb{R}} \int_{\mathbb{R}^+}\varphi(v,I)\left\|Q^{+}(f-g, f+g)\right\|_{L_x^{\infty}}\, d v_{2} d v_{3}dI\right]\, d v_{1}\\
&\equiv II_{1,s}+II_{1,l}.
\end{aligned}
$$
We estimate the small velocity part, we recall from \eqref{app} that
\[
\int_{\left|v_{1}\right|\leq \frac{1}{\varepsilon}}\left(1-\exp\left(-a_2\frac{\varepsilon}{|v_{1}|}\right)\right)dv_1
\leq
2\varepsilon + 4a_2\varepsilon\ln{\frac{1}{\varepsilon}}
\]
to compute
$$
\begin{aligned}
II_{1,s} &=\frac{1}{a_2} \int_{\left|v_{1}\right|\leq \frac{1}{\varepsilon}}\left(1-\exp\left(-a_2\frac{\varepsilon}{|v_{1}|}\right)\right) \times
\left[  \int_{\mathbb{R}} \int_{\mathbb{R}} \int_{\mathbb{R}^+}\varphi(v,I)\left\|Q^{+}(f-g, f+g)\right\|_{L_x^{\infty}}\, d v_{2} d v_{3}dI\right]\, d v_{1}\\
&\leq \frac{1}{a_2} \left[ \int_{\left|v_{1}\right|\leq \frac{1}{\varepsilon}}\left(1-\exp\left(-a_2\frac{\varepsilon}{|v_{1}|}\right)\right)dv_1\right]  \times
\left[   \sup_{P}\int_{P}\int_{\mathbb{R}^+} \varphi(v,I)\left\|Q^{+}(f-g, f+g)\right\|_{L_x^{\infty}} \,dId \pi_{P,v}\right] \\
&\leq  \frac{1}{a_2}\left(2\varepsilon +4a_2\varepsilon\ln{\frac{1}{\varepsilon}}\right) \|Q^+(f+g,f-g)\|_P    \\
&\leq \frac{1}{a_2}\Big(2\varepsilon + 4a_2\varepsilon\ln{\frac{1}{\varepsilon}}\Big)\max\{\pi,2^{2-\gamma}\}\big\{\|f-g\|_0 \cdot \|f+g\|_0 + \|f-g\|_0 \cdot \|f+g\|_{1-\gamma}\big\}  \\
&\leq 2\max\{\pi,2^{2-\gamma}\}\left\{\frac{a_1+a_3}{a_2}\right\}\Big(2\varepsilon + 4a_2\varepsilon\ln{\frac{1}{\varepsilon}}\Big)\|f-g\|_0 .
\end{aligned}
$$
where we used Lemma \ref{planenorm} and  
\[
\|f+g\|_0+\|f+g\|_{1-\gamma}\leq 2\{a_1+a_3\}.
\]
The estimate for the large velocity integral, we use
\[
1-e^{-x}\leq x
\]
and Lemma \ref{nbb} to obtain

\begin{equation*}
\begin{aligned}
 II_{1,l}&= \frac{1}{a_2} \int_{\left|v_{1}\right|\geq \frac{1}{\varepsilon}}\left(1-\exp\left(-a_2\frac{\varepsilon}{|v_{1}|}\right)\right) \times
\left[  \int_{\mathbb{R}} \int_{\mathbb{R}} \int_{\mathbb{R}^+}\varphi(v,I)\left\|Q^{+}(f-g, f+g)\right\|_{L_x^{\infty}}\, d v_{2} d v_{3}dI\right]\, d v_{1}\\
 &\leq \frac{1}{a_2} \int_{\left|v_{1}\right|\geq \frac{1}{\varepsilon}}\left(a_2\frac{\varepsilon}{|v_{1}|}\right) \times
\left[  \int_{\mathbb{R}} \int_{\mathbb{R}} \int_{\mathbb{R}^+}\varphi(v,I)\left\|Q^{+}(f-g, f+g)\right\|_{L_x^{\infty}}\, d v_{2} d v_{3}dI\right]\, d v_{1}\\
&\leq  \varepsilon^{2} \int_{\left|v_{1}\right|>\frac{1}{\varepsilon}} \left[\int_{\mathbb{R}} \int_{\mathbb{R}}\int_{\mathbb{R}^+} \varphi(v,I)\left\|Q^{+}(f-g, f+g)\right\|_{L_x^{\infty}}\, dId v_{2} d v_{3}\right]\,d v_{1} \\
& \leq  \varepsilon^{2} \int_{\mathbb{R}}  \int_{\mathbb{R}} \int_{\mathbb{R}}\int_{\mathbb{R}^+} \varphi(v,I)\left\|Q^{+}(f-g, f+g)\right\|_{L_x^{\infty}} \,dId v_{1}d v_{2} d v_{3}\\
&\leq   \varepsilon^{2} \left\|Q^{+}(f-g,f+g)\right\|_0 \\
&\leq 4\pi\varepsilon^2 \max\left\{\frac{1}{a}, 1\right\} \|f-g\|_0 \cdot \|f+g\|_0 \\
& =8\pi \varepsilon^2 a_1\max\left\{\frac{1}{a}, 1\right\} \|f-g\|_0 
\end{aligned}
\end{equation*}

We combine these estimates $II_{1,s}$ and $II_{1,l}$ to
get the following estimate for $II_1$:
$$
\begin{aligned}
II_1
&\leq 2\max\{\pi,2^{2-\gamma}\} \left\{\frac{a_1+a_3}{a_2}\right\}\Big(2\varepsilon+ 4a_2\varepsilon\ln{\frac{1}{\varepsilon}}\Big)\|f-g\|_0 + 8\pi\varepsilon^2 a_1\max\left\{\frac{1}{a}, 1\right\}  \|f-g\|_0 
\end{aligned}
$$

We now turn to the estimate of $II_2$ in \eqref{ln1}. By the mean value theorem again, there exists $0<\theta<1$ such that
\begin{equation*}
    \begin{aligned}
    II_2&\leq \int_{\mathbb{R}^{3}\times \mathbb{R}^+} \varphi\frac{\varepsilon}{|v_{1}| }\left\| \int_{0}^{x}\left\{\exp\left(-\frac{\varepsilon}{|v_1|}\int_{y}^{x}L(f)\,dz\right)- \exp\left(-\frac{\varepsilon}{|v_1|}\int_{y}^{x}L(g)\,dz\right)\right\}Q^{+}(g,g)\,dy\right\|_{L_x^{\infty}} \, dvdI\\
        &\leq \int_{\mathbb{R}^{3}\times \mathbb{R}^+} \varphi\frac{\varepsilon}{\left|v_{1}\right| }\left\| \int_{0}^{x}  \frac{\varepsilon}{ |v_{1}|}\left(\int_{y}^{x} L(f)-L(g)\, dz\right)\right.\\
        &\hspace{2cm}\left.\times \exp \left(-\frac{\varepsilon}{|v_1|}\int_{y}^{x} \theta L(f) + (1-\theta) L(g)\, dz \right)\,d y\right\|_{L_x^{\infty}}\left\|Q^{+}(g,g)\right\|_{L_x^{\infty}}  \, d vdI\\
        \end{aligned}
        \end{equation*}
        We then recall the upper bound of $L$ from Lemma \ref{lemma6rs} and the lower bound from Lemma \ref{invariance L lower bound} to proceed
        \begin{equation*}
        \begin{aligned}
        II_2&\leq \int_{\mathbb{R}^{3}\times \mathbb{R}^+} \varphi(v,I)\frac{\varepsilon^2}{ |v_1|^2}\left(1+\left(|v|^2+I\right)^{\frac{\gamma}{2}}\right)\left(\frac{4\pi e^a}{a}\right)\left\| f-g\right\|_{0}\\
        &\hspace{2cm}\times \left\|\int_{0}^{x}(x-y) \exp \left(-\frac{\varepsilon}{ |v_{1}|} a_2\left(1+\left(|v|^2+I\right)^{\frac{\gamma}{2}}\right)(x-y) \right)\,dy \right\|_{L_x^{\infty}} \left\|Q^{+}(g,g)\right\|_{L_x^{\infty}} \,  d vdI\\
        &= \int_{\mathbb{R}^{3}\times \mathbb{R}^+} \varphi(v,I)\frac{\varepsilon^2}{ |v_1|^2}\left(1+\left(|v|^2+I\right)^{\frac{\gamma}{2}}\right)\left(\frac{4\pi e^a}{a}\right)\left\| f-g\right\|_{0}\\
        &\hspace{2cm}\times \left\|\int_{0}^{x}\tau \exp \left(-\frac{\varepsilon}{ |v_{1}|} a_2\left(1+\left(|v|^2+I\right)^{\frac{\gamma}{2}}\right)\tau \right)\,dy \right\|_{L_x^{\infty}} \left\|Q^{+}(g,g)\right\|_{L_x^{\infty}} \,  d vdI.
        \end{aligned}
        \end{equation*}
        Note that, in the last line, we set $\tau=x-y$.
        We split $II_2$ into the small velocity integral $II_{2,s}$ and the large velocity integral $II_{2,l}$:
        \begin{equation*}
            \begin{aligned}
II_2&= \frac{4\pi e^a}{a}\left\| f-g\right\|_{0}\left(\int_{|v_1|< \frac{1}{\varepsilon}}+\int_{|v_1|\geq\frac{1}{\varepsilon}}\right)\varphi(v,I)\frac{\varepsilon^2}{|v_1|^2 }\left(1+\left(|v|^2+I\right)^{\frac{\gamma}{2}}\right) \\
&\hspace{2cm}\times\left\|\int_{0}^{x}\tau\exp \left(-\frac{\varepsilon}{ |v_{1}|} \tau a_2\left(1+\left(|v|^2+I\right)^{\frac{\gamma}{2}}\right) \right) \,d\tau\right\|_{L_x^{\infty}}\left\|Q^{+}(g,g)\right\|_{L_x^{\infty}}\,dvdI \\
&\equiv II_{2,s}+II_{2,l}.
    \end{aligned}
\end{equation*}

For $II_{2,s}$, we have
\begin{equation*}
    \begin{aligned}
       \hspace{-2.5cm}II_{2,s}&=\frac{4\pi e^a}{a}\left\| f-g\right\|_{0}\int_{|v_1|<\frac{1}{\varepsilon}} \varphi(v,I)\frac{\varepsilon^2}{|v_1|^2 }\left(1+\left(|v|^2+I\right)^{\frac{\gamma}{2}}\right) \\
       &\hspace{2cm}\times \left\|\int_{0}^{x}\tau\exp \left(-\frac{\varepsilon}{ |v_{1}|} \tau a_2\left(1+\left(|v|^2+I\right)^{\frac{\gamma}{2}}\right) \right) \,d\tau\right\|_{L_x^{\infty}}\left\|Q^{+}(g,g)\right\|_{L_x^{\infty}}\,dvdI \\
       &\leq \frac{4\pi e^a}{a}\left\| f-g\right\|_{0} \int_{|v_1|<\frac{1}{\varepsilon}} \frac{\varepsilon^2}{|v_1|^2 }\left(1+\left(|v|^2+I\right)^{\frac{\gamma}{2}}\right) \left(\int_{0}^{1}\tau\exp \left(-\frac{\varepsilon}{ |v_{1}|} \tau a_2\left(1+\left(|v|^2+I\right)^{\frac{\gamma}{2}}\right) \right) \,d\tau\right) \\
        &\hspace{2cm} \times \left(\int_{\mathbb{R}^2\times \mathbb{R}^+} \varphi(v,I) \left\|Q^{+}(g,g)\right\|_{L_x^{\infty}} \,dv_2dv_3dI\right) \,dv_1 \\
       \end{aligned}
       \end{equation*}
       \begin{equation*}
           \begin{aligned}
        &\leq \frac{4\pi e^a}{ a}\left\|Q^{+}(g,g)\right\|_{P} \left\| f-g\right\|_{0}\\ &\hspace{1cm}\times\int_{|v_1| < \frac{1}{\varepsilon}}\frac{\varepsilon}{|v_{1}| }\left(\int_{0}^{1} \frac{\varepsilon}{|v_1|}\left(1+\left(|v|^2+I\right)^{\frac{\gamma}{2}}\right) \tau \left\{\exp \left(-\frac{\varepsilon}{ 2|v_{1}|} a_2 \tau \left(1+\left(|v|^2+I\right)^{\frac{\gamma}{2}}\right) \right)\right\}^2 \,d\tau\right) \, d v_1 \\
       &\leq \frac{8\pi e^a}{a_2 a}\left\|Q^{+}(g,g)\right\|_{P} \left\| f-g\right\|_{0}\int_{|v_1| < \frac{1}{\varepsilon}}\frac{\varepsilon}{|v_{1}| } \left(\int_{0}^{1}\exp \left\{-\frac{\varepsilon}{ 2|v_{1}|} a_2 \tau \left(1+\left(|v|^2+I\right)^{\frac{\gamma}{2}}\right) \right\}\,d\tau\right) \, d v_1\\
       \end{aligned}
       \end{equation*}
       In the last line, we used the uniform bound of $xe^{-x}\leq1$. Thus,
       \begin{equation*}
           \begin{aligned}
II_{2,s}&\leq \frac{8\pi e^a}{a_2 a} \left\|Q^{+}(g,g)\right\|_{P} \left\| f-g\right\|_{0}\int_{|v_1|< \frac{1}{\varepsilon}} \frac{\varepsilon}{\left|v_{1}\right| } \left(\int_{0}^{1}\exp \left\{-\frac{\varepsilon}{ 2|v_{1}|} \tau a_2\right\}\,d\tau\right) \, d v_1\\
         &= \frac{16\pi e^a}{a_2^2 a} \left\|Q^{+}(g,g)\right\|_{P} \left\| f-g\right\|_{0}\int_{|v_1|< \frac{1}{\varepsilon}} \left(1-\exp \left\{-\frac{\varepsilon}{ 2v_{1}}a_2  \right\}\right) \,  d v_1.
    \end{aligned}
\end{equation*}
Now, we recall \eqref{app} to compute the integral in $v_1$, 
and Lemma \ref{planenorm} to estimate the plain norm of $Q^+$:
\begin{equation*}
    \begin{aligned}
        II_{2,s}
         &\leq \frac{16\pi e^a}{a_2^2 a}\Big(2\varepsilon + 2a_2\varepsilon\ln{\frac{1}{\varepsilon}}\Big)\left\|Q^{+}(g,g)\right\|_{P} \left\| f-g\right\|_{0} \\
         &\leq \frac{16\pi e^a}{a_2^2 a}\Big(2\varepsilon + 2a_2\varepsilon\ln{\frac{1}{\varepsilon}}\Big)\max\{\pi,2^{2-\gamma}\}(\|g\|_0^2 + \|g\|_0 \cdot \|g\|_{1-\gamma} ) \left\| f-g\right\|_{0} \\
         &\leq \max\{\pi,2^{2-\gamma}\} \frac{16\pi e^a}{a}\left\{ \frac{a_1^2+a_1\cdot a_3}{a_2^2}\right\}\Big(2\varepsilon + 2a_2\varepsilon\ln{\frac{1}{\varepsilon}}\Big)\|f-g\|_0 \\
    \end{aligned}
\end{equation*}
We now estimate $II_{2,l}$: 
\begin{equation*}
    \begin{aligned}
    II_{2,l} &=\frac{4\pi e^a}{a}\left\| f-g\right\|_{0}\int_{|v_1|\geq \frac{1}{\varepsilon}}\varphi(v,I)\frac{\varepsilon^2}{|v_1|^2 }\left(1+\left(|v|^2+I\right)^{\frac{\gamma}{2}}\right) \\
&\hspace{1.5cm}\times\left\|\int_{0}^{x}\tau\exp \left(-\frac{\varepsilon}{ |v_{1}|} \tau a_2\left(1+\left(|v|^2+I\right)^{\frac{\gamma}{2}}\right) \right) \,d\tau\right\|_{L_x^{\infty}}\left\|Q^{+}(g,g)\right\|_{L_x^{\infty}}\,dvdI \\
    &\leq  \frac{4\pi e^a}{a} \left\| f-g\right\|_{0} \int_{|v_1| \geq \frac{1}{\varepsilon}} \frac{\varepsilon}{|v_1|}\varphi(v,I) \\
&\hspace{1.5cm}\times \left(\int_{0}^{1}\frac{\varepsilon}{|v_1|}\left(1+\left(|v|^2+I\right)^{\frac{\gamma}{2}}\right) \exp \left\{-\frac{\varepsilon}{ |v_{1}|} \tau a_2\left(1+\left(|v|^2+I\right)^{\frac{\gamma}{2}}\right) \right\}\,d\tau\right)\left\|Q^{+}(g,g)\right\|_{L_x^{\infty}}   \,dvdI. \\
\end{aligned}
\end{equation*}
We then carry out integration with respect to $\tau$ and use the fact that $\frac{1}{|v_1|}<\varepsilon$ to obtain 
\begin{equation*}
    \begin{aligned}
         II_{2,l}&\leq \frac{4\pi e^a}{a_2 a}\varepsilon^2\left\| f-g\right\|_{0}\int_{\mathbb{R}}\int_{\mathbb{R}^{2}\times \mathbb{R}^+} \varphi(v,I) \left(1-\exp \left\{-\frac{\varepsilon}{ |v_{1}|}  a_2\left(1+\left(|v|^2+I\right)^{\frac{\gamma}{2}}\right)\right\}\right)\left\|Q^{+}(g,g)\right\|_{L_x^{\infty}}\,   dvdI \\ 
         &\leq\frac{4\pi e^a}{a_2 a}\varepsilon^2\left\| f-g\right\|_{0}\int \varphi(v,I)\left\|Q^{+}(g,g)\right\|_{L_x^{\infty}}  \, dvdI\\
        &= \frac{4\pi e^a}{a_2 a}\varepsilon^2\left\|Q^{+}(g,g)\right\|_{0}\left\| f-g\right\|_{0} \\
        &\leq \frac{4\pi e^a}{a_2 a}\varepsilon^2\left\| f-g\right\|_{0} \cdot 4\pi \max\left\{ \frac{1}{a}, 1\right\} \|f\|_0^2 \\
        &\leq \frac{16\pi^2 e^a}{a}\max\left\{ \frac{1}{a}, 1\right\} \left( \frac{a_1^2}{a_2}\right) \varepsilon^2 \|f-g\|_0.
    \end{aligned}
\end{equation*}
We combine the estimate for $II_{2,s}$ and $II_{2,l}$ to obtain
\begin{align*}
    II_2\leq \max\{\pi,2^{2-\gamma}\}\frac{16\pi e^a}{a}\left\{ \frac{a_1^2+a_1\cdot a_3}{a_2^2}\right\}\left(2\varepsilon + 2a_2\varepsilon\ln{\frac{1}{\varepsilon}}\right)\|f-g\|_0 + 16\pi^2\frac{ e^a}{a}\max\left\{ \frac{1}{a}, 1\right\} \left( \frac{a_1^2}{a_2}\right) \varepsilon^2 \|f-g\|_0.
\end{align*}
Combining all the estimates $I$ and $II$ to get
        \begin{align*}
        \|\Psi^+(f)-\Psi^+(g)\|_0 &\leq C \varepsilon \|f-g\|_0 + 2\max\{\pi,2^{2-\gamma}\} \left\{\frac{a_1+a_3}{a_2}\right\}\Big(2\varepsilon+ 4a_2\varepsilon\ln{\frac{1}{\varepsilon}}\Big)\|f-g\|_0 \\
        &+ 8\pi\varepsilon^2 a_1\max\left\{\frac{1}{a}, 1\right\}  \|f-g\|_0 + \frac{16\pi^2 e^a}{a}\max\left\{ \frac{1}{a}, 1\right\} \left( \frac{a_1^2}{a_2}\right) \varepsilon^2 \|f-g\|_0\\
        &+\max\{\pi,2^{2-\gamma}\}\frac{16\pi e^a}{a}\left\{ \frac{a_1^2+a_1\cdot a_3}{a_2^2}\right\}\left(2\varepsilon + 2a_2\varepsilon\ln{\frac{1}{\varepsilon}}\right)\|f-g\|_0\\
        &\leq C(\varepsilon) \|f-g\|_0.
        \end{align*}
    Similarly for $\Psi^-$,
        \begin{align*}
            \|\Psi^-(f)-\Psi^-(g)\|_0 \leq C(\varepsilon)\|f-g\|_0.
        \end{align*}
    Therefore,
    \begin{align*}
        \|\Psi(f)-\Psi(g)\|_0 \leq C(\varepsilon)\|f-g\|_0.
    \end{align*}
     We can conclude that the mapping $\Psi$ is contractive for sufficiently small $\varepsilon$. This completes the proof.
    \end{proof}

\section{Appendix}
 Various collision kernels are available in the literature.  \cite{9,milana}, that also satisfy \eqref{reversibility} and  \eqref{boundednessofB} by choosing certain functions $\Phi$ and $\Psi$ satisfying \eqref{symmetry}. 
\begin{enumerate}
\item The total energy model
\begin{equation*}\label{model3}
	\mathcal{B}\left(v, v_{*}, I, I_{*}, r, R,\sigma\right)=\left(\frac{1}{4}|v-v_*|^2+I+I_*\right)^{\frac{\gamma}{2}},
\end{equation*}
which is equivalent to the model \eqref{eqm} and obtained with the choice
	$$\Phi_{\gamma}(r,R)=\frac{1}{2^{\frac{\gamma}{2}+1}}, \qquad \Psi_{\gamma}(r,R)=1.$$
	\item kinetic and microscopic internal energy detached model
	\begin{equation*}\label{model2}
		\mathcal{B}\left(v, v_{*}, I, I_{*}, r, R,\sigma\right)=R^{\frac{\gamma}{2}}\left|v-v_{*}\right|^{\gamma}+(1-R)^{\frac{\gamma}{2}} \left({I+I_{*}}\right)^{\frac{\gamma}{2}},
	\end{equation*}
	obtained for
	$$\Phi_{\gamma}(r,R)=\min\{R,(1-R)\}^{\frac{\gamma}{2}}, \text{  and \;} \Psi_{\gamma}(r,R)=\max\{R,(1-R) \}^{\frac{\gamma}{2}}.$$
	
\item 	kinetic and particles' microscopic internal energies detached model
	\begin{equation*}\label{model1}
		\begin{aligned}
			\mathcal{B}\left(v, v_{*}, I, I_{*}, r, R,\sigma\right)=R^{\frac{\gamma}{2}}\left|v-v_{*}\right|^{\gamma}+\left(r(1-R) {I}\right)^{\frac{\gamma}{2}}+\left((1-r)(1-R) {I_{*}}\right)^{\frac{\gamma}{2}}
		\end{aligned}
	\end{equation*}
	obtained for
	$$\Phi_{\gamma}(r,R)=\min\{R,(1-R)\}^{\frac{\gamma}{2}}\min\{r,(1-r)\}^{\frac{\gamma}{2}}, \text{  and \;} \Psi_{\gamma}(r,R)=2^{1-\frac{\gamma}{2}}\max\{R,1-R \}^{\frac{\gamma}{2}}.$$
\end{enumerate}
with the following upper and lower bounds introduced in \cite{milana}
	\begin{equation}\label{boundednessofB}
	    \;{\Phi_{\gamma}}(r,R)\;\Big({|v-v_*|}^{{\gamma}}+(I+I_*)^{\frac{\gamma}{2}}\Big)\leq \bb\leq \;\Psi_{\gamma}(r,R)\Big({|v-v_*|}^{{\gamma}}+(I+I_*)^{\frac{\gamma}{2}}\Big)
	\end{equation}
where $\gamma\geq 0$. In addition, ${\Phi}_{\gamma}$ and  $\Psi_{\gamma}$ are positive functions such that
\begin{equation*}
	{\Phi}_{\gamma}\leq \Psi_{\gamma},
\end{equation*}
	and
\begin{equation}\label{symmetry}
	\Phi_{\gamma}(r,R)=\Phi_{\gamma}(1-r,R), \quad \Psi_{\gamma}(r,R) =\Psi_{\gamma}(1-r,R).
\end{equation}
\newline\newline
\noindent{\bf Acknowledgement}\newline
\noindent Seok-Bae Yun was supported by the National Research Foundation of Korea(NRF) grant funded by
the Korean government(MSIT) (RS-2023-NR076676).

\bibliographystyle{plain}
\bibliography{bibliography}

\noindent\textit{Department of Mathematics, Sungkyunkwan University, Suwon 440-746, Republic of Korea} \\
Email: \texttt{hkn4463@skku.edu}

\vspace{1em}

\noindent\textit{Department of Mechanical Engineering, Eindhoven University of Technology, Vector Building,
5631 BN Eindhoven, The Netherlands} \\
Email: \texttt{m.shahine@tue.nl}

\vspace{1em}

\noindent\textit{Department of Mathematics, Sungkyunkwan University, Suwon 440-746, Republic of Korea} \\
Email: \texttt{sbyun01@skku.edu}

\end{document}